\documentclass[11pt, reqno]{amsart}

\usepackage{amsmath, amssymb, amsfonts, amstext, verbatim, amsthm}
\usepackage{graphicx}
\usepackage[all]{xy}
\usepackage{color}
\usepackage{enumitem}
\usepackage[modulo]{lineno}
\usepackage{hyperref}
\usepackage{cite}

\usepackage{fullpage}

\theoremstyle{plain}
\newtheorem{thm}{Theorem}[section]
\newtheorem{lem}[thm]{Lemma}

\newtheorem{cor}[thm]{Corollary}

\newtheorem{prop}[thm]{Proposition}

\newtheorem{amplif}[thm]{Amplification}
\newtheorem{ansatz}[thm]{Ansatz}

\theoremstyle{definition}
\newtheorem{rem}[thm]{Remark}
\newtheorem{defn}[thm]{Definition}
\newtheorem{ex}[thm]{Example}

\makeatletter
\newcommand{\customlabel}[2]{%
\protected@write \@auxout {}{\string \newlabel {#1}{{#2}{}}}}
\makeatother

\newcommand{\bA}{\mathbf{A}}
\newcommand{\cA}{\mathcal{A}}

\newcommand{\cB}{\mathcal{B}}
\newcommand{\fB}{\mathfrak{B}}

\newcommand{\bC}{\mathbf{C}}
\newcommand{\cC}{\mathcal{C}}

\newcommand{\cD}{\mathcal{D}}

\newcommand{\cE}{\mathcal{E}}

\newcommand{\bG}{\mathbf{G}}

\newcommand{\cH}{\mathcal{H}}

\newcommand{\cI}{\mathcal{I}}

\newcommand{\bL}{\mathbf{L}}
\newcommand{\cL}{\mathcal{L}}

\newcommand{\cN}{\mathcal{N}}
\newcommand{\fN}{\mathfrak{N}}

\newcommand{\cO}{\mathcal{O}}

\newcommand{\bP}{\mathbf{P}}

\newcommand{\bR}{\mathbf{R}}

\newcommand{\cT}{\mathcal{T}}

\newcommand{\fU}{\mathfrak{U}}

\newcommand{\cV}{\mathcal{V}}

\newcommand{\fX}{\mathfrak{X}}

\newcommand{\bZ}{\mathbf{Z}}

\renewcommand{\bR}{{\mathbb R}}
\renewcommand{\bC}{{\mathbb C}}
\renewcommand{\bZ}{{\mathbb Z}}

\newcommand{\op}[1]{\!\!\mathop{\rm ~#1}\nolimits}

\renewcommand{\bP}{\mathbb{P}}
\renewcommand{\bA}{\mathbb{A}}
\newcommand{\pt}{\ast}

\newcommand{\sh}[1]{\mathcal{#1}}
\newcommand{\st}[1]{\mathfrak{#1}}
\newcommand{\lie}[1]{\mathfrak{#1}}

\newcommand{\oneps}{1-PS}

\renewcommand{\bG}{\mathbb{G}}

\newcommand{\inner}[1]{{\underline{#1}}}

\newcommand{\dual}{\vee}

\newcommand{\X}{\st{X}}
\renewcommand{\S}{\st{S}}
\newcommand{\Z}{\st{Z}}
\newcommand{\V}{\st{V}}

\newcommand{\coh}[1]{\op{Coh}(#1)}
\newcommand{\qcoh}[1]{\op{QCoh}(#1)}

\newcommand{\radj}[1]{\beta_{\geq #1}}
\newcommand{\ladj}[1]{\beta_{< #1}}
\newcommand{\G}{\mathbf{G}}
\newcommand{\D}{\op{D}}
\newcommand{\bdot}{{\begin{picture}(4,4)\put(2,3){\circle*{1.5}}\end{picture}}}

\newcommand{\colim}{\underrightarrow{\lim}}

\newcommand{\parr}{\dashrightarrow}
\newcommand{\gd}{\mathbb{D}}

\newcommand{\coNorm}{{\mathfrak N}^\dual}

\newcommand{\Gm}{\mathbb{G}_m}

\begin{document}

\title{The derived category of a GIT quotient}
\author{Daniel Halpern-Leistner}

\begin{abstract}
Given a quasiprojective algebraic variety with a reductive group action, we describe a relationship between its equivariant derived category and the derived category of its geometric invariant theory quotient. This generalizes classical descriptions of the category of coherent sheaves on projective space and categorifies several results in the theory of Hamiltonian group actions on projective manifolds.

This perspective generalizes and provides new insight into examples of derived equivalences between birational varieties. We provide a criterion under which two different GIT quotients are derived equivalent, and apply it to prove that any two generic GIT quotients of an equivariantly Calabi-Yau projective-over-affine manifold by a torus are derived equivalent.

\bigskip
\begin{flushright}
\textit{Dedicated to Ernst Halpern, who inspired my scientific pursuits.}
\end{flushright}
\end{abstract}

\maketitle
\tableofcontents



\section{Introduction}

We describe a relationship between the derived category of equivariant coherent sheaves on a smooth projective-over-affine variety, $X$, with a linearizable action of a reductive group, $G$, and the derived category of coherent sheaves on a GIT quotient, $X//G$, of that action. Our main theorem connects three classical circles of ideas:
\begin{itemize}
\item Serre's description of quasicoherent sheaves on a projective variety in terms of graded modules over its homogeneous coordinate ring,
\item Kirwan's theorem that the canonical map $H_G^\ast (X) \to H^\ast(X//G)$ is surjective \cite{Ki84}, and
\item the ``quantization commutes with reduction'' theorem from geometric quantization theory equating $h^0(X,\sh{L})^G$ with $h^0(X//G,\sh{L})$ when the linearization $\sh{L}$ descends to the GIT quotient \cite{Te00}.
\end{itemize}

Let us recall the construction of a GIT quotient. A $G$-linearized ample line bundle $\sh{L}$ defines an open semistable locus $X^{ss} \subset X$, defined to be the complement of the base locus of invariant global sections of $\sh{L}^k$ for $k\gg 0$. We denote the quotient stack $\X = X/G$, and in this paper the term ``GIT quotient" will refer to the quotient stack $\X^{ss} = X^{ss}/G$, as opposed to the coarse moduli space of $X^{ss}/G$.

In order to state the main theorem, we will need to recall the equivariant ``Kempf-Ness (KN) stratification" of $X \setminus X^{ss}$ by connected locally-closed subvarieties \cite{DH98}. We formally define a KN stratification and discuss its properties in Section \ref{sect_main_result}. The stratification is determined by a set of \emph{distinguished one-parameter subgroups}, $\lambda_i : \Gm \to G$, and open subvarieties of the fixed locus of $\lambda_i$ denoted $\sigma_i : Z_i \hookrightarrow X$. We will also define integers $\eta_i \geq 0$ in \eqref{eqn:define_eta}. Because $Z_i$ is fixed by $\lambda_i$, the restriction of an equivariant coherent sheaf $\sigma_i^\ast F$ is graded with respect to the weights of $\lambda_i$.

We denote the bounded derived category of coherent sheaves on $\X$ by $\D^b(\X)$, and likewise for $\X^{ss}$.\footnote{On a technical note, all of the categories in this paper will be pre-triangulated dg-categories, so $\D^b(\X)$ denotes a dg-enhancement of the triangulated category usually denoted $\D^b(\X)$. However, all of the results will be statements that can be verified on the level of homotopy categories, such as semiorthogonal decompositions and equivalences of categories, so we will often write proofs on the level of the underlying triangulated category.} Restriction gives an exact dg-functor $i^\ast : \D^b(X/G) \to \D^b(X^{ss}/G)$, and in fact any bounded complex of equivariant coherent sheaves on $X^{ss}$ can be extended equivariantly to $X$. The main result of this paper is the construction of a \emph{functorial} splitting of $i^\ast$.
\begin{thm}[categorical Kirwan surjectivity, preliminary statement]\label{thm_kirwan_surj_prelim}
Let $X$ be a smooth projective-over-affine variety with a linearized action of a reductive group $G$, and let $\X = X/G$. Specify an integer $w_i$ for each KN stratum of the unstable locus $\X \setminus \X^{ss}$. Define the full subcategory of $\D^b(\X)$
$$\G_w := \left\{ F^\bdot \in \D^b(\X) \left| \lambda_i \text{-weights of }\cH^\ast (L\sigma_i^\ast F^\bdot ) \text{ lie in } [w_i, w_i + \eta_i ) \right. \right\}$$
Then the restriction functor $i^\ast : \G_w \to \D^b(\X^{ss})$ is an equivalence of categories.
\end{thm}
\begin{rem}
The general version, described in Section \ref{sect_main_result}, identifies $\G_w$ as piece of a semiorthogonal decomposition of $\D^b(\X)$, and it applies to any (possibly singular) stack $\X$ such that $\X \setminus \X^{ss}$ admits a KN stratification (Definition \ref{def_KN_stratification}) satisfying Properties \hyperref[property_L_plus]{(L+)} and \hyperref[property_A]{(A)}.
\end{rem}

The simplest example of Theorem \ref{thm_kirwan_surj_prelim} is familiar to many mathematicians: projective space $\bP(V)$ can be thought of as a GIT quotient of $V/\Gm$. Theorem \ref{thm_kirwan_surj_prelim} identifies $\D^b(\bP(V))$ with the full triangulated subcategory of the derived category of equivariant sheaves on $V$ generated by $\cO_V(q),\cdots,\cO_V(q+\dim V-1)$. In particular the semiorthogonal decompositions described in Section \ref{sect_baric_decomp} refine and provide an alternative proof of Beilinson's theorem that the line bundles $\cO_{\bP(V)}(1), \ldots, \cO_{\bP(V)}(\dim V)$ generate $\D^b(\bP(V))$.

Serre's theorem deals with the situation in which $G =\Gm$, $X$ is an affine cone, and the unstable locus consists only of the cone point -- in other words one is studying a connected, positively graded $k$-algebra $A$. The category of quasicoherent sheaves on $\op{Proj}(A)$ can be identified with the full subcategory of the category of graded $A$-modules graded in degree $\geq q$ for any fixed $q$. This classical result has been generalized to noncommutative $A$ by M. Artin and J. J. Zhang \cite{AZ94}. D. Orlov studied the derived category and the category of singularities of such an algebras in great detail in \cite{Or09}, and much of the technique of the proof of Theorem \ref{thm_kirwan_surj_prelim} derives from that paper.

In the context of equivariant K\"{a}hler geometry, one can consider Theorem \ref{thm_kirwan_surj_prelim} as a categorification of Kirwan surjectivity. Kirwan surjectivity applies to topological $K$-theory in addition to cohomology \cite{HL07}, and one immediate corollary of Theorem \ref{thm_kirwan_surj_prelim} is an analogous statement for algebraic $K$-theory.
\begin{cor}
The restriction map on algebraic $K$-theory $K_i(\X) \to K_i(\X^{ss})$ is surjective.
\end{cor}
\noindent In a follow up paper, we will describe more precisely how to recover cohomological Kirwan surjectivity from \ref{thm_kirwan_surj_prelim} as well.

The fully faithful embedding $\D^b(\X^{ss}) \subset \D^b(\X)$ of Theorem \ref{thm_kirwan_surj_prelim} and the more precise semiorthogonal decomposition of Theorem \ref{thm_kirwan_surj} correspond, via Orlov's analogy between derived categories and motives \cite{Or05}, to the claim that the motive $\X^{ss}$ is a summand of $\X$. Via this analogy, the results of this paper bear a strong formal resemblance to the motivic direct sum decompositions of homogeneous spaces arising from Bia{\l}ynicki-Birula decompositions \cite{Br05}. However, the precise analogue of Theorem \ref{thm_kirwan_surj_prelim} would pertain to the equivariant motive $X/G$, whereas the results of \cite{Br05} pertain to the nonequivariant motive $X$.

The ``quantization commutes with reduction'' theorem from geometric quantization theory relates to the fully-faithfulness of the functor $i^\ast$. The original conjecture of Guillemin and Sternberg, that $\dim H^0(X/G,\sh{L}^k) = \dim H^0(X^{ss}/G,\sh{L}^k)$, has been proven by several authors, but the most general version was proven by Teleman in \cite{Te00}. He shows that the canonical restriction map induces an isomorphism $R\op{\Gamma}(X/G,\sh{V}) \to R\op{\Gamma}(X^{ss}/G,\sh{V})$ for any equivariant vector bundle such that $\sh{V}|_{Z_i}$ is supported in weight $>-\eta_i$. If $\sh{V}_1$ and $\sh{V}_2$ are two vector bundles such that the $\lambda_i$-weights of $\sh{V}|_{Z_i}$ lie in $[w_i,w_i+\eta_i)$, then the fact that $R\op{Hom}^\bdot_\X(\sh{V}_1,\sh{V}_2) \to R\op{Hom}^\bdot_{\X^{ss}} (\sh{V}_1|_{\X^{ss}}, \sh{V}_2|_{\X^{ss}})$ is an isomorphism is precisely Teleman's quantization theorem applied to $\sh{V}_2 \otimes \sh{V}_1^\dual \simeq R\inner{\op{Hom}}(\sh{V}_1,\sh{V}_2)$.

In Section \ref{sect_var_GIT}, we apply Theorem \ref{thm_kirwan_surj_prelim} to construct new examples of derived equivalences and embeddings resulting from birational transformations, as conjectured by Bondal \& Orlov \cite{BO02}. The $G$-ample cone in $NS_G^1(X)$ has a decomposition into convex conical chambers within which the GIT quotient $\X^{ss}(\cL)$ does not change \cite{DH98}, and $\X^{ss}(\cL)$ undergoes a birational transformation as $[\cL]$ crosses a wall between chambers. Categorical Kirwan surjectivity provides a general approach to constructing derived equivalences between the quotients on either side of the wall: in some cases both quotients can be identified by Theorem \ref{thm_kirwan_surj_prelim} with the \emph{same} subcategory of $\D^b(X/G)$. This principle is summarized in Ansatz \ref{ansatz_flop}.

For a certain class of wall crossings, \emph{balanced wall crossings} (Definition \ref{def_balanced_wall_crossing}), there is a simple criterion for when one gets an equivalence or an embedding in terms of the weights of $\omega_X|_{Z_i}$. When $G=T$ is abelian, all codimension-1 wall crossings are balanced. In particular we are able to prove that any two generic torus quotients of an equivariantly Calabi-Yau variety are derived equivalent. For nonabelian $G$, we consider a slightly larger class of \emph{almost balanced wall crossings}. We produce derived equivalences for flops which excise a Grassmannian bundle over a smooth variety and replace it with the dual Grassmannian bundle, recovering recent work of Will Donovan and Ed Segal \cite{Do11,DS12}.

Finally, in Section \ref{sect_further_applications} we investigate applications of Theorem \ref{thm_kirwan_surj} beyond smooth quotient stacks $X/G$. We identify a criterion under which Property \hyperref[property_L_plus]{(L+)} holds for a KN stratification, and apply it to hyperk\"{a}hler reductions. We also explain how the Morita theory of \cite{BFN10} recovers derived Kirawn surjectivity for certain complete intersections and derived categories of singularities (equivalently categories of matrix factorizations) ``for free" from the smooth case.

The inspiration for Theorem \ref{thm_kirwan_surj_prelim} were the grade restriction rules for the category of boundary conditions for B-branes of Landau-Ginzburg models studied by Hori, Herbst, and Page in \cite{HHP08}, as interpreted mathematically by Segal \cite{Se11}. The essential idea of splitting was present in that paper, but the analysis was only carried out for a linear action of $\Gm$, and the category $\G_w$ was identified in an ad-hoc way. The main contribution of this paper is showing that the splitting can be globalized and applies to arbitrary $X/G$ as a categorification of Kirwan surjectivity, and that the categories $\G_w$ arise naturally via the semiorthogonal decompositions to be described in the next section.

\subsection{Author's note}

I would like to thank my PhD adviser Constantin Teleman for introducing me to his work \cite{Te00}, and for his support and useful comments throughout this project. I would like to thank Daniel Pomerleano for many enlightening conversations, and for explaining how to recover derived categories of singularities using Morita theory. I'd like to thank Anatoly Preygel for useful conversations about derived algebraic geometry and for carefully reviewing section \ref{sect_further_applications}. Finally, I'd like to thank Yujiro Kawamata for suggesting that I apply my methods to hyperk\"{a}hler reduction and flops of Grassmannian bundles, and Kentaro Hori for carefully reviewing my work and discovering some mistakes in the first version of this paper.

The problems studied in this paper overlap greatly with the work \cite{BFK12}, although the projects were independently conceived and carried out. I learned about \cite{BFK12} at the January 2012 Conference on Homological Mirror Symmetry at the University of Miami, where the authors presented a method for constructing equivalences between categories of matrix factorizations of toric LG models. In the finished version of their paper, they also treat the general VGIT for smooth quotients $X/G$, and present several new applications. Here we work in slightly more generality and emphasize the categorification of Kirwan surjectivity, as well as some applications to hyperk\"{a}hler quotients. We hope that the different perspectives brought to bear on the subject will be useful in elucidating further questions.

\subsection{Notation}
We will work over an algebraically closed field of characteristic zero, although we expect that some version of these results hold in characteristic $p$ as well. We will denote schemes with Roman font, $X,S,Z$, etc. Throughout the paper, most of the schemes we encounter will have a specified action of an algebraic group, in which case we denote the corresponding quotient stacks $\X,\S, \Z$ respectively. Sheaves will be denoted with Roman font as well, and sheaves of algebras will be denoted $\cO,\cA$, etc.

For stacks and schemes, we will work with several variants of the derived category. We use cohomological indexing conventions, so $F^\bdot = \cdots \to F^i \to F^{i+1} \to \cdots$. $\D^b(\X)$ will denote the derived category of $\cO_\X$ modules with coherent, bounded cohomology. $\op{Perf}(\X) \subset \D^b(\X)$ will denote the subcategory of perfect complexes. $\D_{qc}(\X)$ will denote the unbounded derived category of complexes with quasicoherent cohomology, and $\D^+(\X)$ (resp. $\D^-(\X)$) will denote the category with coherent, bounded below (resp. above) cohomology. The category $\D^b_\S(\X) \subset \D^b(\X)$ will denote the category of complexes whose cohomology sheaves are set-theoretically supported on $\S$ (after restricting to a local atlas on $\X$).

Unless otherwise specified, all functors between derived categories are derived. For instance if $j : \S \to \X$ is a morphism of stacks, $j^\ast$ denotes the derived pullback $Lj^\ast$. At times we will revert to the classical notation $Lj^\ast, R\pi_\ast$, etc. to prevent confusion with the non-derived versions of the functors. We underline functors to denote the sheaf-theoretic version, so $\inner{\op{Hom}}(E^\bdot,F^\bdot)$ denotes the derived sheaf-Hom, and $\inner{\op{\Gamma}}_S(F^\bdot)$ denotes the derived subsheaf with supports.

We will sometimes make use of the standard $t$-structures on $\D_{qc}(\X)$ and its variants, which we denote with a superscript $\D_{qc}(\X)^{\leq p}$ and $\D_{qc}(\X)^{\geq p}$. We will also encounter subcategories of $\D^b(\X)$ defined by various ``weight conditions," which we denote using subscripts such as $\D^b(\X)_{\geq w}$ and $\D^b(\X)_{<w}$.


\section{The main theorem}
\label{sect_main_result}

First we shall review the properties of the Kempf-Ness (KN) stratification of the unstable locus in geometric invariant theory and establish notation. Then we will formulate and lay out the proof of Theorem \ref{thm_kirwan_surj_prelim}. The key technical results that comprise the proof will be treated separately in Section \ref{sect_baric_decomp}.

\subsection{Equivariant stratifications in GIT}
First let us recall the construction of the KN stratification of a projective-over-affine variety $X \subset \bP^n \times \bA^m$ invariant with respect to the action of a reductive group $G$. We let $\cL := \cO_X(1)$ with a chosen linearization. We fix an inner product on the cocharacter lattice of $G$ which is invariant under the Weil group. This allows us to define a conjugation-invariant norm $|\lambda|$ of any one-parameter subgroup (\oneps) such that $|\lambda| > 0$ for any nontrivial \oneps.

For any \oneps, $\lambda : \Gm \to G$, the \emph{blade} associated to a connected component $Z \subset X^\lambda$ is defined to be the locally closed subvariety
\begin{equation} \label{eqn_define_attracting_set}
Y_{\lambda , Z} := \{ x \in X | \lim \limits_{t \to 0} \lambda(t) \cdot x \in Z \} \subset X.
\end{equation}
The natural projection map $\pi : Y_{\lambda,Z} \to Z$ is affine with connected fibers, hence $Y_{\lambda,Z}$ is connected. $g \in G$ acts on the set of such pairs $(\lambda, Z)$ by $g \cdot (\lambda,Z) = (g \lambda g^{-1},g Z)$. Up to this action we can assume that $\lambda$ factors through a fixed choice of maximal torus of $G$, and the set of possible $Z$ appearing in such a pair is finite.

One constructs the KN stratification iteratively by selecting a pair $(\lambda_i,Z^\ast)$ which maximizes the numerical invariant
$$\mu(\lambda,Z) := \frac{-1}{|\lambda|} \op{weight}_\lambda \cL|_Z \in \bR$$
among those $(\lambda,Z)$ for which $Z$ is not contained in the union of the previously defined strata. One defines the open subvariety $Z_i \subset Z^\ast$ to consist of those points not lying on previously defined strata, and the blade $Y_i := \pi^{-1} (Z_i) \subset Y_{\lambda_i,Z^\ast}$. Finally we define the new stratum to be $S_i = G \cdot Y_i$. We repeat this process until there are no pairs with $\mu(\lambda,Z) >0$ in which $Z$ is not contained in the union of the previously defined strata.

By construction, the stratification is preordered by the value of the numerical invariant $\mu$, but for clarity we will always choose a refinement to a total ordering by integers.\footnote{Alternatively, one could index the stratification by the values of $\mu$ by defining $Y_\mu$ to be the union of all of the blades with a fixed numerical invariant and specifying different distinguished \oneps's for each connected component. The arguments and results of this paper are essentially unmodified by allowing such disconnected strata.} It is a non-trivial fact that the iterative procedure outlined above produces strata such that
$$\bar{S_i} \setminus S_i \text{ and } \bar{Y_i} \setminus Y_i \subset \bigcup_{j > i} S_j,$$
so that the Hilbert-Mumford procedure leads to an ascending sequence of $G$-equivariant open subvarieties $X^{ss}:= X_0 \subset X_1 \subset \cdots \subset X_n = X$. It is evident that the stratification of $\bP^n \times \bA^m$ induces the stratification of $X$, and in fact the Hilbert-Mumford procedure can fail to produce such a stratification if $X$ is not projective-over-affine.

For each $i$ we define the subgroup $L_i \subset G$ to be the set of $l \in G$ which centralize $\lambda_i$ and such that $l (Z_i) \subset Z_i$.\footnote{In general any $g \in G$ which commutes with $\lambda_i$ must permute the connected components of $X^{\lambda_i}$; however if $G$ is connected, then the centralizer of any \oneps\ is connected, so the condition that $l(Z_i) \subset Z_i$ is unnecessary. In this case if $\lim_{t\to 0} \lambda_i(t) p \lambda_i(t)^{-1}$ exists, then it must be in $L_i$.} Likewise, we define the parabolic subgroup $P_i \subset G$ of all $p \in G$ such that $\lambda_i(t) p \lambda_i(t)^{-1}$ has a limit in $L_i$ as $t\to 0$. $L_i$ is a Levi component of $P_i$, so we have the semidirect product sequence
\begin{equation} \label{eqn_parabolic_def}
\xymatrix{ 1 \ar[r] & U_i \ar[r] & P_i \ar[r] & L_i \ar[r] \ar@/_/[l] & 1 }
\end{equation}
where $U_i \subset P_i$ is the unipotent radical. The locally closed subvarieties $S_i$ enjoy some special properties with respect to $\lambda_i$ (see \cite{Ki84}, \cite{DH98} and the references therein):
\begin{enumerate}[label=(S\arabic{*})]

\item $Z_i$ is fixed by $\lambda_i$ and equivariant with respect to $L_i$. $Y_i$ is $P_i$-equivariant, and the canonical projection $\pi_i : Y_i \to Z_i$ given by $$\pi_i : x \mapsto \lim_{t \to 0} \lambda_i(t) \cdot x$$ is algebraic, affine, and equivariant in the sense that it intertwines the canonical quotient homomorphism $P_i \to L_i$.\label{property_S_1}\\

\item The canonical map $G \times_{P_i} Y_i \to G \cdot Y_i =: S_i$ is an isomorphism. \label{property_S_2}\\

\item The conormal sheaf $\coNorm_{S_i} X= \sh{I}_{S_i} / \sh{I}_{S_i}^2$ restricted to $Z_i$ has positive weights with respect to $\lambda_i$. \label{property_S_3} \\
\end{enumerate}

\begin{rem}
Properties \ref{property_S_1} and \ref{property_S_3} hold for any subvariety of the form $G \cdot Y_{\lambda,Z}$, where $Y_{\lambda,Z}$ is a blade defined as in \eqref{eqn_define_attracting_set}, so \ref{property_S_2} is the only property essential to the strata arising in GIT. Note also that when $G$ is a torus, then $L_i=P_i=G$ and $Y_i = S_i$ for all $i$, so \ref{property_S_2} is automatic, and the description of the stratification is much simpler.
\end{rem}

Due to the iterative construction of the KN stratification, it will suffice to analyze a single \emph{closed} stratum $S \subset X$. Our proof of the main theorem will be a simple induction from the case of a single closed stratum. We will simplify notation by dropping the index $i$ everywhere.

Property \ref{property_S_2} implies that as stacks the natural map $Y / P \to S/G$ is an equivalence, and we can therefore identify the category of $G$-equivariant quasicoherent sheaves on $S$ with the category of $P$-equivariant quasicoherent sheaves on $Y$. Explicitly, the equivalence is given by restricting a quasicoherent sheaf to $Y$ and remembering the $P$-equivariant structure. We will also use $j$ to denote the closed immersion of stacks $Y/P \hookrightarrow X/G$.

If we let $P$ act on $Z$ via the projection $P \to L$, then \ref{property_S_1} lets us identify $Y = \inner{\op{Spec}}_{Z} (\sh{A})$ where $\sh{A} = \sh{O}_{Z} \oplus \bigoplus_{i < 0} \sh{A}_i$ is a coherently generated $P$-equivariant $\sh{O}_{Z}$-algebra, nonpositively graded with respect to the weights of $\lambda$. Thus quasicoherent sheaves on the quotient stack $\S$ can further be identified with quasicoherent $P$-equivariant $\cA$-modules on $Z$. We will return to this description and study the category $\D^b(\S)$ in detail in Section \ref{sect_baric_decomp}.

\begin{defn}[KN stratification] \label{def_KN_stratification}
Let $X$ be a quasiprojective variety with a linearizable action of a reductive group $G$. A closed \emph{Kempf-Ness (KN) stratum} is a closed subvariety $S \subset X$ such that there is a $\lambda$ and an open-and-closed subvariety $Z \subset X^\lambda$ satisfying properties \ref{property_S_1}-\ref{property_S_3}. We will introduce standard names for the morphisms
\begin{equation} \label{eqn_label_arrows}
\xymatrix{Z \ar[r]^-{\sigma} & Y \subset S \ar[r]^-{j} \ar@/^/[l]^-{\pi} & X}
\end{equation}
If $X$ is not smooth along $Z$, we make the following technical hypothesis:
\begin{itemize}
\item[($\dagger$)] There is a $G$-equivariant closed immersion $X \subset X^\prime$ and a KN stratum $S^\prime \subset X^\prime$ such that $S$ is a union of connected components of $S^\prime \cap X$ and $X^\prime$ is smooth in a neighborhood of $Z^\prime$.
\end{itemize}
Let $X^u \subset X$ be a closed equivariant subvariety. A collection of locally closed subvarieties $S_i \subset X^u$, $i=1,\ldots,n$, will be called a \emph{KN stratification} if $X^u = \bigcup_i S_i$ and $S_i \subset X - \bigcup_{j > i} S_j$ is a closed KN stratum for all $i$.
\end{defn}

\begin{rem}
The technical hypothesis is only used for the construction of Koszul systems in Section \ref{sect_koszul_systems}. It is automatically satisfied for the GIT stratification of a projective-over-affine variety.
\end{rem}

In order to state our main theorem (Theorem \ref{thm_kirwan_surj} below), we will introduce two additional hypotheses on the KN strata:
\begin{itemize}
\setlength{\itemindent}{5pt}
\setlength{\itemsep}{5pt}
\item[($A$)] $\pi : Y \to Z$ is a locally trivial bundle of affine spaces, and \label{property_A}
\item[($L+$)] The derived restriction along the closed immersion $\sigma : Z \hookrightarrow S$ of the relative cotangent complex, $L\sigma^\ast \bL^\bdot_{S / X}$, has nonnegative weights w.r.t. $\lambda$. \label{property_L_plus}
\end{itemize}
We will use the construction of the cotangent complex in characteristic $0$ as discussed in \cite{Ma09}.

\begin{ex} \label{ex:Serre}
Let $X \subset \bP^n$ be a projective variety with homogeneous coordinate ring $A$. The affine cone $\op{Spec} A$ has $\Gm$ action given by the nonnegative grading of $A$ and the unstable locus is $Z=Y=S=$ the cone point. $\cO_S$ can be resolved as a semi-free graded dg-algebra over $A$, $(A[x_1,x_2,\ldots], d) \to \cO_S$ with generators of positive weight. Thus $\bL^\bdot_{S/Z} = \cO_S \otimes \Omega^1_{A[x_1,\ldots] / A}$ has positive weights. The Property \hyperref[property_A]{(A)} is automatic, as $Y=Z$.
\end{ex}

\begin{ex} \label{ex:cotangent_fail} Consider the graded ring $k[x_1,\ldots,x_n,y_1,\ldots,y_m] / (f)$ where the $x_i$ have positive degrees, the $y_i$ have negative degrees, and $f$ is a homogeneous polynomial such that $f(0)=0$. This corresponds to a linear action of $\Gm$ on an equivariant hypersurface $X_f$ in the affine space $\bA^n_x \times \bA^m_y$. Assume that we have chosen the linearization such that $S = \{0\} \times \bA^m_y \cap X_f$. One can compute
$$\bL^\bdot_{S/X_f} =  \left\{ \begin{array}{l} (\cO_S dx_1 \oplus \cdots \oplus \cO_S dx_n) [1], \text{ if } f \notin (x_1,\ldots,x_n) \\
(\cO_S f \to \cO_S dx_1 \oplus \cdots \oplus \cO_S dx_n)[1] \text{ if }f \in (x_1,\ldots,x_n) \end{array} \right.$$
where in the latter case the map is determined by $f \mapsto df \mod dy_1,\ldots, dy_m$.

Because the $x_i$ have positive degree, one sees that Property \hyperref[property_L_plus]{(L+)} fails if and only if $f \in (x_1,\ldots,x_n)$ and $\deg(f) <0$. Furthermore, Property \hyperref[property_A]{(A)} amounts to $S$ being an affine space, which happens iff  $\deg f \geq 0$, so that $S = \bA^m_y$, or $\deg f <0$ and the reduction of $f$ modulo $(x_1,\ldots,x_n)$ is linear in the $y_i$.
\end{ex}

\begin{rem}
In Example \ref{ex:cotangent_fail}, one could flip the linearization so that $S^\prime = \bA^n_x \times \{0\} \cap X_f$ is unstable, with distinguished \oneps\ $\lambda(t) = t^{-1}$. In order for Properties \hyperref[property_A]{(A)} and \hyperref[property_L_plus]{(L+)} to hold in both linearizations, there are only two possibilities: either $\deg f = 0$ or $\deg f < 0$ (resp. $\deg f>0$) and the reduction of $f$ modulo $(x_1,\ldots,x_n)$ (resp. $(y_1,\ldots,y_m)$) is linear.
\end{rem}

As Example \ref{ex:cotangent_fail} shows, Properties \hyperref[property_A]{(A)} and \hyperref[property_L_plus]{(L+)} can be fairly restrictive. Fortunately these properties hold automatically when $X$ is smooth.
\begin{lem} \label{lem:smooth_strata}
Let $X$ be smooth in a neighborhood of $Z$. Then $Z$, $Y$, and $S$ are smooth, and $\pi :Y \to Z$ is a bundle of affine spaces as in \hyperref[property_A]{(A)}. Furthermore $X$ is smooth in a neighborhood of $S$, and Property \hyperref[property_L_plus]{(L+)} holds automatically.
\end{lem}
\begin{proof}
The fact that $Z$ is smooth and $\pi : Y \to Z$ is a bundle of affine spaces (hence smooth) is Bia{\l}ynicki-Birula's theorem. $S$ is smooth by \ref{property_S_2}. Any $G$-equivariant open neighborhood of $Z$ contains $S$, hence $X$ is smooth in a neighborhood of $S$. It follows that $j : S \hookrightarrow X$ is a regular embedding, so $\bL^\bdot_{S / X} \simeq \coNorm_S X [1]$ is locally free on $S$, and \hyperref[property_L_plus]{(L+)} follows from \ref{property_S_3}.
\end{proof}

In addition to Example \ref{ex:Serre}, we will study other singular examples where \hyperref[property_A]{(A)} and \hyperref[property_L_plus]{(L+)} hold in Section \ref{sect_further_applications}, where we apply our results to hyperk\"{a}hler reductions.

\subsection{Statement and proof of the main theorem}

As discussed in the introduction, we will consider a quasiprojective variety $X$ with a linearizable action of a reductive group $G$ and an open subvariety $X^{ss} \subset X$. We will use the symbol $\X$ to denote the quotient stack $X/G$, and likewise for $\X^{ss}$. We let $\{\S_i\}_{i = 1,\ldots,N}$ be a KN stratification (Definition \ref{def_KN_stratification}) of $\X^u = \X \setminus \X^{ss}$. As the statement of Theorem \ref{thm_kirwan_surj_prelim} indicates, we will construct a splitting of $\D^b(\X) \to \D^b(\X^{ss})$ by identifying a subcategory $\G_w \subset \D^b(\X)$ that is mapped isomorphically onto $\D^b(\X^{ss})$. In fact we will identify $\G_w$ as the middle factor in a large semiorthogonal decomposition of $\D^b(\X)$.

For each KN stratum, let $\sigma_i : Z_i \hookrightarrow S_i$ and $j_i : S_i \hookrightarrow X$ denote the respective inclusions. When it is clear from context, we will use $\sigma_i$ rather than $j_i \circ \sigma_i$ to denote the inclusion $Z_i \hookrightarrow X$. Recall the shriek pullback functor $j_i^! : \D^+(\X) \to \D^+(\S_i)$ which assigns $j_i^! F^\bdot = \inner{\op{Hom}}_\fU(\cO_{\S_i},F^\bdot|_\fU)$ regarded as an $\cO_{\S_i}$ module, where $\fU$ is an open substack containing $\S_i$ as a closed substack. 


\begin{defn}
For each KN stratum, choose an integer $w_i \in \bZ$, and denote the corresponding function $w : \{0,\ldots,N\} \to \bZ$. Define the full subcategories of $\D^b(\X)$:
\begin{align*}
\D^b_{\X^u}(\X)_{\geq w} &:= \{ F^\bdot \in \D^b_{\X^u}(\X) \left| \forall i , \lambda_i \text{-weights of } \cH^\ast (\sigma_i^\ast F^\bdot) \text{ are } \geq w_i \right. \} \\
\D^b_{\X^u}(\X)_{<w} &:= \{ F^\bdot \in \D^b_{\X^u}(\X) \left| \forall i, \lambda_i \text{-weights of } \cH^\ast (\sigma_i^\ast j_i^! F^\bdot) \text{ are } < w_i \right. \} \\
\G_w &:= \left\{ F^\bdot \left| \begin{array}{l} \forall i, \lambda_i \text{-weights of } \cH^\ast (\sigma_i^\ast F^\bdot) \text{ are } \geq w_i \text{, and} \\ \lambda_i \text{-weights of } \cH^\ast(\sigma_i^\ast j_i^! F^\bdot) \text{ are } < w_i \end{array} \right. \right\}
\end{align*}
We refer to the conditions defining $\G_w$ as a grade restriction rule.\footnote{In a large class of examples, the paper \cite{HHP09} defines subcategories of $\D^b(\X)$ which are described explicitly in terms of generating sets of invertible sheaves. Our $\G_w$ agree with those studied in the examples of \cite{HHP09}, so we have adopted the terminology ``grade restriction rule'' even though our definition of $\G_w$ is different.}
\end{defn}

When $X$ is smooth in a neighborhood of each $Z_i$, one can characterize $\G_w$ and $\D^b_{\X^u}(\X)_{<w}$ in terms of $\sigma_i^\ast F^\bdot$, which avoids the reference to the stratum $\S_i$. This will be useful when we apply categorical Kirwan surjectivity to a variation of GIT quotient in Section \ref{sect_var_GIT}. In Section \ref{sect_describe_categories} we will discuss further ways to describe these categories. 

By Lemma \ref{lem:smooth_strata}, $\bL^\bdot_{\S_i/\X}[-1] \simeq \coNorm_{S_i} X$ is a locally free sheaf when $X$ is smooth in a neighborhood of each $Z_i$. In this case $\det(\coNorm_{S_i} X)$ is an equivariant invertible sheaf and its restriction to $Z_i$ is concentrated in a single nonnegative weight with respect to $\lambda_i$ (it is $0$ iff $\coNorm_{S_i} X = 0$). We define
\begin{align} \label{eqn:define_eta}
\eta_i &:= \op{weight}_{\lambda_i} \det(\coNorm_{S_i} X)|_{Z_i} \\
&= \op{weight}_{\lambda_i} \det(\coNorm_{Y_i} X)|_{Z_i} - \op{weight}_{\lambda_i} \det(\lie{g}_{\lambda_i > 0}) \nonumber
\end{align} 
The second equality follows from three facts: the conormal sequence $0 \to \coNorm_{S_i} X \to \coNorm_{Y_i} X \to \coNorm_{Y_i} S_i \to 0$; Property \ref{property_S_2} implies that $\coNorm_{Y_i} S_i \simeq (\lie{g}_{\lambda_i < 0})^\dual$; and $\det(\lie{g}_{\lambda_i < 0})^\dual$ has the same $\lambda_i$ weight as $\det(\lie{g}_{\lambda_i > 0})$ because $\det(\lie{g})$ has weight $0$.

\begin{lem} \label{lem:smooth_strata_2}
If $X$ is smooth in a neighborhood of each $Z_i$, then
\begin{align*}
\D^b_{\X^u}(\X)_{<w} &:= \{ F^\bdot \in \D^b_{\X^u}(\X) \left| \forall i, \lambda_i \text{-weights of } \cH^\ast (\sigma_i^\ast F^\bdot) \text{ are } < w_i + \eta_i \right. \} \\
\G_w &:= \left\{ F^\bdot \left| \forall i, \lambda_i \text{-weights of } \cH^\ast (\sigma_i^\ast F^\bdot) \text{ lie in } [w_i, w_i+\eta_i) \right. \right\}
\end{align*}
\end{lem}

\begin{proof}
By Lemma \ref{lem:smooth_strata}, the inclusion $j_i : \S_i \hookrightarrow \X$ is a regular embedding, so $j_i^! \simeq \det (\fN_{\S_i/\X}) \otimes j_i^\ast$. Therefore, for $F^\bdot \in \D^b(\X)$, the weights of $\sigma_i^\ast j_i^! F^\bdot$ are $<w$ if and only if the weights of $\det (\fN_{\S_i/\X})|_{\Z_i} \otimes \sigma_i^\ast(F^\bdot)$ are $<w$. The Lemma now follows from the fact that the line bundle $\det(\fN_{\S_i / \X})$ must be concentrated in a single weight, which we have defined to be $- \eta_i$.
\end{proof}

We denote a \emph{semiorthogonal decomposition} of a triangulated category $\cD$ by full triangulated subcategories $\cA_i$ as $\cD = \langle \cA_n, \ldots, \cA_1 \rangle$. This means that all morphisms from objects in $\cA_i$ to objects in $\cA_j$ are zero for $i<j$, and for any object $E \in \cD$ there is a sequence $0 = E_0 \to E_1 \to \cdots \to E_n = E$ with $\op{Cone}(E_{i-1} \to E_i) \in \cA_i$, which is necessarily unique and functorial.\footnote{There are two additional equivalent ways to characterize a semiorthogonal decomposition: 1) the inclusion of the full subcategory $\cA_i \subset \langle \cA_i , \cA_{i-1}, \ldots, \cA_1 \rangle$ admits a left adjoint (is left admissible) $\forall i$, or 2) the subcategory $\cA_i \subset \langle \cA_n, \ldots, \cA_i \rangle$ is right admissible $\forall i$. In some contexts one also requires that each $\cA_i$ be admissible in $\cD$, but we will not require this here. See \cite{BO95} for further discussion of semiorthogonal decompositions.} In our applications $\cD$ will always be a pre-triangulated dg-category, in which case if $\cA_i \subset \cD$ are full pre-triangulated dg-categories then we will abuse the notation $\cD = \langle \cA_n, \ldots, \cA_1 \rangle$ to mean that there is a semiorthogonal decomposition of homotopy categories, in which case $\cD$ is uniquely identified with the gluing of the $\cA_i$. We can now state our main theorem.

\begin{thm}[derived Kirwan surjectivity]\label{thm_kirwan_surj}
Assume that each $\S_i$ satisfies Properties \hyperref[property_L_plus]{(L+)} and \hyperref[property_A]{(A)}. Then are semiorthogonal decompositions
\begin{gather}
\D^b_{\X^u}(\X) = \langle \D^b_{\X^u}(\X)_{< w}, \D^b_{\X^u}(\X)_{\geq w}\rangle \\
\D^b(\X) = \langle \D^b_{\X^u}(\X)_{< w}, \G_w , \D^b_{\X^u}(\X)_{\geq w}\rangle
\end{gather}
and the restriction functor $i^\ast : \G_w \to \D^b(\X^{ss})$ is an equivalence of categories. Furthermore we have $\op{Perf}_{\X^u}(\X)_{\geq v} \otimes^L \D^b_{\X^u}(\X)_{\geq w} \subset \D^b_{\X^u}(\X)_{\geq v+w}$.
\end{thm}

The technical heart of this result is Theorem \ref{thm_main_semi_decomp} below, which is the special case of Theorem \ref{thm_kirwan_surj} in which $N = 1$, so $\X^u$ consists of a single closed KN stratum $\S \subset \X$. Section \ref{sect_baric_decomp} consists of the proof of Theorem \ref{thm_main_semi_decomp}, but here we observe how the general statement follows from the case of a single stratum.

\begin{proof}
We proceed by induction on the ascending sequence of open substacks, with $\fX_0 = \fX^{ss}$ and $\fX_i := \fX_{i-1} \cup \S_i$ for $i = 1,\ldots,N$. We also define $\X^u_n := \S_1 \cup \cdots \cup \S_n \subset \X_n$. We proceed by induction on $N$, with the base case where $N=0$ and $\X^{ss} = \X$.

Consider the closed KN stratum $\S_n \subset \X_n$. Theorem \ref{thm_main_semi_decomp} provides the semiorthogonal decomposition
$$\D^b(\X_n) = \langle \D^b_{\S_n}(\X_n)_{<w}, \tilde{\G}_w, \D^b_{\S_n}(\X_n)_{\geq w} \rangle$$
where $\tilde{\G}_w \subset \D^b(\fX_n)$ consists of complexes satisfying the grade restriction rule along $\S_n$ only. The category $\tilde{\G}_w$ is mapped isomorphically onto $\D^b(\fX_{n-1})$ via restriction. Using this isomorphism and the inductive hypothesis, we obtain a $5$-term semiorthogonal decomposition of $\D^b(\fX_n)$
\begin{equation} \label{eqn:inductive_semidecomp}
\langle \D^b_{\S_n}(\X_n)_{<w}, \D^b_{\X^u_{n-1}}(\X_{n-1})_{<w}, \G^n_w, \D^b_{\X^u_{n-1}}(\X_{n-1})_{\geq w} , \D^b_{\S_n}(\X_n)_{\geq w} \rangle
\end{equation}
Where $\G^n_w$ is the subcategory defined by the grade restriction rules for the KN stratification $\S_1,\ldots,\S_n \subset \X_n$.

Note that an $F^\bdot$ is supported on $\X^u_n$ if and only if its restriction to $\fX_{n-1}$ is supported on $\fX^u_{n-1}$. It follows that the first two terms in the semiorthogonal decomposition \eqref{eqn:inductive_semidecomp} generate the subcategory $\D^b_{\X^u_n}(\X_n)_{<w}$, and the last two terms generate $\D^b_{\X^u_n}(\X_n)_{\geq w}$. Furthermore $\D^b_{\X^u_n} (\X_n)_{<w}$ and $\D^b_{\X^u_n}(\X_n)_{\geq w}$ generate $\D^b_{\X^u_n}(\X_n)$. The theorem follows by induction.
\end{proof}

The semiorthogonal decomposition in this theorem can be refined further using ideas of Kawamata \cite{Ka06}, and Ballard, Favero, Katzarkov \cite{BFK12}. Consider the stack $\Z_n := Z_n / L_n$ and the canonical projection $\pi_n : \S_n \to \Z_n$, and let $\D^b(\Z_n)_w$ denote the full subcategory of complexes whose cohomology is concentrated in weight $w$ with respect to $\lambda_n$. Corollary \ref{cor_infinite_semidecomp} implies that for each $\X_n$ appearing in the proof of Theorem \ref{thm_kirwan_surj}, $(j_n)_\ast \pi_n^\ast : \D^b(\Z_n)_w \to \D^b(\X_n)$ is a fully faithful embedding. Thus as an immediate consequence of Corollary \ref{cor_infinite_semidecomp} and the proof of Theorem \ref{thm_kirwan_surj}, we have
\begin{amplif} \label{amplif:decompose_unstable}
Assume that for each stratum, the restriction of the relative cotangent complex, $\sigma^\ast_n \bL^\bdot_{\S_n / \X}$, has strictly positive weights with respect to $\lambda_n$. Then the categories $\D^b_{\X^u}(\X)_{\geq w}$ and $\D^b_{\X^u}(\X)_{< w}$ admit semiorthogonal decompositions
\begin{align*}
\D^b_{\X^u}(\X)_{\geq w} = \langle &\D^b(\Z_1)_{w_1},\D^b(\Z_1)_{w_1+1},\ldots, \\
&\D^b(\Z_2)_{w_2}, \D^b(\Z_2)_{w_2+1}, \ldots, \ldots , \\
&\D^b(\Z_N)_{w_N}, \D^b(\Z_N)_{w_N+1} \ldots \rangle \\
\D^b_{\X^u}(\X)_{< w} = \langle &\ldots, \D^b(\Z_N)_{w_N-2}, \D^b(\Z_N)_{w_N-1}, \\
&\ldots, \D^b(\Z_2)_{w_2-2}, \D^b(\Z_2)_{w_2-1}, \ldots , \\
& \ldots,\D^b(\Z_1)_{w_1-2},\D^b(\Z_1)_{w_1-1} \rangle
\end{align*}
which can be combined with Theorem \ref{thm_kirwan_surj} to obtain an infinite semiorthogonal decomposition of $\D^b(\X)$.
\end{amplif}

\begin{rem}
By an infinite semiorthogonal decomposition we mean that the subcategories are semiorthogonal to one another, and every object can be constructed via a finite sequence of mapping cones from objects in the subcategories.
\end{rem}

\begin{rem}
If we let $L^\prime_i := L_i / \lambda_i(\Gm)$, then $\Z_i \to Z_i / L^\prime_i$ is a $\Gm$-gerbe. The pullback functor identifies $\D^b(\Z_i)_0$ with $\D^b(Z_i / L^\prime_i)$, and $\D^b(\Z_i)_w$ is the derived category of bounded coherent sheaves twisted by the $w^{th}$ power of this $\Gm$-gerbe.
\end{rem}

In the remainder of this section we discuss two example applications of Theorem \ref{thm_kirwan_surj} in situations of interest.

\begin{ex}[Derived category of a Grassmannian]
The Grassmanian of $d$ dimensional subspaces of $\bA^n$ can be obtained as a GIT quotient of the space, $V$, of $n\times d$ matrices by $GL_d$ acting by $g \cdot M := M g^{-1}$ for $g \in GL_d$ and $M \in V$. In this case
$$\lambda_i = \op{diag} (1,\ldots,1,\underbrace{t,\ldots,t}_{i\text{ times}} ), \qquad Z_i = Y_i = \left\{ [ \underbrace{\ast}_{\stackrel{n \times (n-i),}{\text{full rank}}} | 0 ] \right\},$$
and the stratum $S_i$ consists of all matrices of rank $n-i$. One can choose weights $w_i$ such that Kapranov's exceptional collection corresponds to vector bundles of the form $\cO_V(W)$, where $W$ is an irreducible representation of $GL_d$ satisfying the grade restriction rules. However, one can also choose the $w_i$ such that there are no sheaves of the form $\cO_V \otimes W$ in $\G_w$. See Example \ref{ex_grassmannian} for a closely related example and a more detailed discussion of the stratification.
\end{ex}

\begin{ex}[Elaboration of Serre's theorem]
Let $Z$ be a quasiprojective scheme and $\cA = \bigoplus_{i\geq 0} \cA_i$ a coherently generated sheaf of graded algebras over $Z$, with $\cA_0 = \cO_Z$. Letting $X = \inner{\op{Spec}}_Z(\cA)$, the grading defines a $\Gm$-action on $X$, and we take the unstable stratum to be $j : Z \hookrightarrow X$.

This is a slight generalization of Example \ref{ex:Serre}, and the argument for why \hyperref[property_A]{(A)} and \hyperref[property_L_plus]{(L+)} hold applies in this more general setting. Theorem \ref{thm_kirwan_surj} gives a precise relationship between $\D^b(X/\Gm) = \D^b(\op{gr}-\cA)$, the derived category of sheaves of coherent graded $\cA$-modules, and $\D^b(X - S / \Gm) = \D^b(\inner{\op{Proj}}_Z (\cA))$.\footnote{We should take $\inner{\op{Proj}}_Z(\cA)$ to mean the DM stack $(\inner{\op{Spec}}_Z(\cA) - Z) / \Gm$. This will only be a scheme if $\cA$ is generated in degree $1$, so that $\bG_m$ acts freely on $\inner{\op{Spec}}_Z(\cA) - Z$.} There is an infinite semiorthogonal decomposition,
$$\D^b(\op{gr}-\cA) = \langle \ldots, \D^b(Z)_{w-1}, \G_w, \D^b(Z)_{w}, \D^b(Z)_{w+1},\ldots \rangle$$
where $\D^b(Z)_w$ denotes the subcategory generated by $j_\ast \D^b(Z) \otimes \cO_X(-w)$, and
$$\G_w = \left\{ F^\bdot \in \D^b(X/\Gm) \left| \begin{array}{l} \cH^\ast(j^\ast F^\bdot) \text{ has weights } \geq w, \text{ and} \\ \cH^\ast(j^! F^\bdot) \text{ has weights } < w \end{array} \right. \right\}$$
and the restriction functor $\G_w \to \D^b(\inner{\op{Proj}}_Z \cA)$ is an equivalence.
\end{ex}

\subsection{Explicit constructions of the splitting and integral kernels}

Theorem \ref{thm_kirwan_surj} states that the restriction functor $\G_w \to \D^b(\X^{ss})$ is an equivalence of dg-categories. We now discuss the inverse functor a bit more explicitly. We start with a single closed KN stratum $\S \subset \X$, which satisfies \hyperref[property_A]{(L+)}, and we assume that the $\lambda$-weights of $\sigma^\ast \bL^\bdot_{\S/\X}$ are strictly positive, so that Amplification \ref{amplif:decompose_unstable} holds. We let $\V := \X-\S$.

Given $G^\bdot \in \D^b(\V)$, it is always possible to choose a complex $F^\bdot \in \D^b(\X)$ such that $F^\bdot |_\V \simeq G^\bdot$. By Lemma \ref{lem_bounded_weights}, there is a unique $a\leq b$ with $a$ maximal and $b$ minimal such that the $\lambda$-weights of $\sigma^\ast j^! F^\bdot$ are $< b$, and the $\lambda$-weights of $\sigma^\ast F^\bdot$ are $\geq a$. Let $E^\bdot \in \D^b(\Z)_a$ be the lowest nonvanishing weight subcomplex of $\sigma^\ast F^\bdot$,\footnote{The object $j^\ast F^\bdot$ will not be cohomologically bounded, but its lowest weight space will be.} then the semiorthogonal decomposition of Proposition \ref{prop_baric_S_flat} and Remark \ref{rem_extending_SOD} imply that there is a canonical morphism $j^\ast F^\bdot \to \pi^\ast E^\bdot$. The corresponding morphism $F^\bdot \to j_\ast \pi^\ast (E^\bdot)$ induces an isomorphism between the complexes concentrated in weight $\leq a$ after applying $\sigma^\ast$.

Assume that $a < w$, and define a new object $(F^\bdot)^\prime$ with $(F^\bdot)^\prime|_\V \simeq G^\bdot$ by the exact triangle
$$ (F^\bdot)^\prime \to F^\bdot \to j_\ast (\pi^\ast E^\bdot) \parr .$$
The $\lambda$-weights of $\sigma^\ast (F^\bdot)^\prime$ are $\geq a+1$. Furthermore, the $\lambda$-weights of $\sigma^\ast j^! j_\ast \pi^\ast (E^\bdot)$ are $\leq a$, so $\sigma^\ast j^! (F^\bdot)^\prime$ will still have $\lambda$-weights $<b$ unless $a = b$, in which case $\sigma^\ast j^! (F^\bdot)^\prime$ will at least have $\lambda$-weights $<w$. Iterating this procedure, we will eventually have an object $F^\bdot$ such that $F^\bdot|_{\V} \simeq G^\bdot$, the $\lambda$-weights of $\sigma^\ast F^\bdot$ are $\geq w$, and the $\lambda$-weights of $\sigma^\ast j^! F^\bdot$ are $< \max(w,b)$.

If $b > w$, then by an entirely dual procedure we let $E^\bdot$ be the subcomplex of $\sigma^\ast j^! F^\bdot = \sigma^\ast j^! F^\bdot$ in weight $b$ and consider the cone of a canonically defined map $(F^\bdot)^\prime = \op{Cone}(j_\ast \pi^\ast E^\bdot \to F^\bdot)$. The $\lambda$-weights of $\sigma^\ast (F^\bdot)^\prime$ will be $\geq w$ still, but $\sigma^\ast j^! (F^\bdot)^\prime$ will have $\lambda$-weights $<b-1$. We can repeat this procedure until we finally have an $F^\bdot \in \G_w$ such that $F^\bdot|_\V \simeq G^\bdot$. Theorem \ref{thm_kirwan_surj} now implies that this $F^\bdot$ is the unique lift of $G^\bdot$ lying in $\G_w$.

When there are multiple strata, we must repeat this procedure to lift the object inductively over each stratum as in the proof of Theorem \ref{thm_kirwan_surj}. This process is quite complicated, especially when there are multiple strata. Fortunately, in many examples it suffices to directly construct such a lift for a single universal example in order to obtain an integral kernel for the functor $\D^b(\X^{ss}) \to \G_w \subset \D^b(\X)$.

Let us assume for simplicity that $\X^{ss}$ is a smooth and proper stack, so that
\begin{itemize}
\item the diagonal $\X^{ss} \to \X^{ss} \times \X^{ss}$ is finite (it is affine because $\X^{ss}$ is a global quotient and proper by the assumption that $\X^{ss}$ is separated), and
\item the push forward $\pi_\ast : \D_{qc}(\X^{ss}) \to \D_{qc}(\op{Spec} k)$ preserves bounded coherent objects (this requires characteristic $0$).
\end{itemize}
Under these hypotheses the diagonal sheaf $\cO_{\Delta}$, which is the push forward of $\cO_{\X^{ss}}$ along $\X^{ss} \to \X^{ss} \times \X^{ss}$, lies in $\D^b(\X^{ss} \times \X^{ss})$. Consider the product $\X^{ss} \times \X = (X^{ss} \times X) / (G \times G)$, and the open substack $\X^{ss} \times \X^{ss}$ whose complement $\X^{ss} \times \X^u$ has the KN stratification $\X^{ss} \times \S_i$. It is immediate that Properties \hyperref[property_A]{(A)} and \hyperref[property_L_plus]{(L+)} hold for this KN stratification, so one can uniquely extend the diagonal sheaf, $\cO_{\Delta}$, to a complex, $\tilde{\cO}_\Delta$, in the subcategory $\G_w$ with respect to this stratification.

Consider the integral functor $\D^b(\X^{ss}) \to \D^b(\X)$ with kernel $\tilde{\cO}_\Delta$:
$$\Phi : F^\bdot \mapsto (p_2)_\ast (\tilde{\cO}_\Delta \otimes p_1^\ast(F^\bdot))$$
Because $\X^{ss}$ is smooth, $F^\bdot$ is perfect, and the object $\tilde{\cO}_\Delta \otimes p_1^\ast F^\bdot$ is bounded with coherent cohomology. It follows that $(p_2)_\ast (\tilde{\cO}_\Delta \otimes p_1^\ast F^\bdot) \in \D^b(\X)$.

\begin{lem}
For all $F^\bdot \in \D^b(\X^{ss})$, $\Phi(F^\bdot) \in \G_w$.
\end{lem}

\begin{proof}
We consider the fiber square
$$\xymatrix{\X^{ss} \times \Z_i \ar[r]^{\sigma_i^\prime} \ar[d]^{p_2^\prime} & \X^{ss} \times \X \ar[d]^{p_2} \\ \Z_i \ar[r]^{\sigma_i} & \X }$$
By base change $\sigma_i^\ast \Phi(F^\bdot) \simeq (p_2^\prime)_\ast (\sigma_i^\prime)^\ast (\tilde{\cO}_\Delta \otimes p_1^\ast F^\bdot)$. We have $p_1^\ast F^\bdot \in \D^b(\X^{ss} \times \Z_i)_0$, where the weight is with respect to the distinguished \oneps\ on the right factor, and $\tilde{\cO}_\Delta \in \D^b(\X^{ss} \times \X)_{\geq w_i}$ by hypothesis. It follows that $(\sigma_i^\prime)^\ast (\tilde{\cO}_\Delta) \otimes (\sigma^\prime_i)^\ast p_1^\ast F^\bdot \in \D^b(\X^{ss} \times \Z_i)_{\geq w_i}$ and thus $\sigma_i^\ast \Phi(F^\bdot) \in \D^b(\Z_i)_{\geq w_i}$. As similar argument shows that $\sigma_i^\ast j^! F^\bdot \in \D_{qc}(\Z_i)_{<w_i}$.
\end{proof}

Furthermore, $\Phi(F^\bdot)|_{\X^{ss}}$ is canonically isomorphic to the push forward to $\X^{ss}$ of $\cO_\Delta \otimes p_1^\ast F^\bdot$. Because the integral functor with kernel $\cO_\Delta$ is just the identity, we have a canonical isomorphism $\Phi(F^\bdot)|_{\X^{ss}} \simeq F^\bdot$. Therefore, $\Phi : \D^b(\X^{ss}) \to \G_w \subset \D^b(\X)$ is the inverse of the restriction equivalence $\G_w \to \D^b(\X^{ss})$.

\section{Homological structures on the unstable strata}
\label{sect_baric_decomp}

In this section we will study in detail the homological properties of a single closed KN stratum $\S := S/G \subset \X$ as in Definition \ref{def_KN_stratification}. We will also let $\V$ denote the open complement $\V = \X - \S$ and study the relationship between $\D^b(\X)$ and $\D^b(\V)$.

Our main theorem is Theorem \ref{thm_main_semi_decomp}, which is the key to the inductive proof of Theorem \ref{thm_kirwan_surj}. In fact, Theorem \ref{thm_main_semi_decomp} is just a summary of several results throughout this section. Before launching into the technical content, we give an overview of the ideas which follow.

Our main conceptual tool is the notion of a baric decomposition, which was introduced and used to construct `staggered' $t$-structures on equivariant derived categories of coherent sheaves \cite{AT11}.

\begin{defn}
A \emph{baric decomposition} of a triangulated category $\cD$ is a family of semiorthogonal decompositions $\cD = \langle \cD_{<w}, \cD_{\geq w} \rangle$ such that $\cD_{\geq w} \supset \cD_{\geq w+1}$, or equivalently $\cD_{<w} \subset \cD_{<w+1}$, for all $w$. The baric decomposition is \emph{bounded} if $\cD = \bigcup_{v,w} (\cD_{\geq w} \cap \cD_{<v})$. If $\cD \subset \D_{qc}(\X)$ for some stack $\X$, then we say that the baric decomposition is \emph{multiplicative} if $E^\bdot \otimes \cD_{\geq w} \subset \cD_{\geq w+v}$ whenever $E^\bdot \in \cD_{\geq v} \cap \op{Perf}(\X)$.
\end{defn}

Although the connection with GIT was not explored in the original development of the theory, baric decompositions arise naturally in this context. In Proposition \ref{prop_baric_S_flat} we establish a multiplicative baric decomposition on $\D^b(\S)$ when $Y \to Z$ is flat. Because $\lambda(\Gm)$ is central in $L$ and stabilizes $Z$, objects in $\D^b(Z/L)$ decompose canonically as a direct sum of weight eigen-complexes. This is no longer true in $\D^b(\S)$, but the baric decomposition assigns to each object a canonical sub-quotient in weight $w$.

\begin{ex}
Consider the case when $\S \simeq \op{Spec} k[x_1,\ldots,x_n] / \Gm$, where the $\Gm$ action is determined by a choice of a negative grading on each $x_i$. Then $\cD_{<w} = \D^b(\S)_{<w}$ is the triangulated category generated by graded modules whose nonzero weight spaces have weight $<w$.
\end{ex}

After some technical preparations in Sections \ref{sect_cotangent_complex} and \ref{sect_koszul_systems}, we show that when $\S \subset \X$ satisfies Property \hyperref[property_L_plus]{(L+)}, the categories generated by the pushforward of $\D^b(\S)_{\geq w}$ and $\D^b(\S)_{<w}$ in $\D^b(\X)$ remain semiorthogonal. This provides (Proposition \ref{prop_baric_support_S}) a multiplicative baric decomposition $\D^b_\S(\X) = \langle \D^b_\S(\X)_{< w}, \D^b_\S(\X)_{\geq w} \rangle$, where $\D^b_\S(\X)$ denotes the derived category of complexes of coherent sheaves on $\X$ whose restriction to $\V = \X-\S$ is acyclic. One nice consequence of this machinery is a generalization of Teleman's ``quantization commutes with reduction theorem'' (Theorem \ref{thm_quantization}).

In Section \ref{sect_main_semi_decomp} we use the baric decomposition of $\D^b_\S(\X)$ to analyze the category $\D^b(\X)$ itself. On the level of derived categories of quasicoherent sheaves, the inclusion $\D_{\S,qc}(\X) \subset \D_{qc}(X)$ always admits a right adjoint $R\inner{\op{\Gamma}}_\S$. However for $F^\bdot \in \D^b(\X)$, the object $R \inner{\op{\Gamma}}_\S(F^\bdot)$ no longer has coherent cohomology.

Our main observation, which holds assuming Property \hyperref[property_A]{(A)}, is that for $F^\bdot \in \D^b(\X)$, it is possible, informally speaking, to keep only the piece of $R\inner{\op{\Gamma}}_\S(F^\bdot)$ whose homology has weight $\geq w$, and that this new object $\radj{w} R\inner{\op{\Gamma}}_\S(F^\bdot)$ has bounded coherent cohomology. Thus we construct a right adjoint to the inclusion $\D^b_\S(\X)_{\geq w} \subset \D^b_\S(\X)$, and dually we construct a left adjoint for the inclusion of $\D^b_\S(\X)_{<w}$. 

It follows that if we \emph{define} $\G_w \subset \D^b(\X)$ to be the right orthogonal to $\D^b_\S(\X)_{\geq w}$ and left orthogonal to $\D^b_\S(\X)_{<w}$, then there is a $3$-term semiorthogonal decomposition
$$\D^b(\X) = \langle \D^b_\S(\X)_{<w},\G_w,\D^b_\S(\X)_{\geq w} \rangle .$$
The Quantization Theorem, \ref{thm_quantization}, says precisely that $\G_w$ is mapped fully-faithfully to $\D^b(\V)$ under restriction, and in fact the restriction functor gives an isomorphism $\G_w \simeq \D^b(\V)$.

\subsection{Quasicoherent sheaves on $\S$}

Recall the structure of a KN stratum \eqref{eqn_label_arrows} and the associated parabolic subgroup \eqref{eqn_parabolic_def}. By Property \ref{property_S_1}, $\S := S/G \simeq Y/P$ via the $P$-equivariant inclusion $Y \subset S$, so we will identify quasicoherent sheaves on $\S$ with $P$-equivariant quasicoherent $\cO_Y$-modules. Furthermore, we will let $P$ act on $Z$ via the projection $P \to L$. Again by Property \ref{property_S_1}, we have $Y / P = \inner{\op{Spec}}_{Z} (\sh{A}) / P$, where $\sh{A}$ is a coherently generated graded $\cO_Z$-algebra with $\sh{A}_i = 0$ for $i>0$, and $\sh{A}_0 = \sh{O}_Z$. Thus we have identified quasicoherent sheaves on $\S$ with $P$-equivariant quasicoherent $\sh{A}$-modules on $\Z^\prime := Z/P$.

\begin{rem}
The stack $\Z := Z/L$ is perhaps more natural than the stack $\Z^\prime$. The projection $\pi : Y \to Z$ intertwines the respective $P$ and $L$ actions via $P \to L$, hence we get a projection $\S \to \Z$. Unlike the map $\S \to \Z^\prime$, this projection admits a section $Z/L \to Y/P$. In other words, the projection $\cA \to \cA_0 = \cO_Z$ is $L$-equivariant, but not $P$-equivariant. We choose to work with $\Z^\prime$, however, because the map $\S \to \Z$ is not representable, so the description of quasicoherent sheaves on $\S$ in terms of ``quasicoherent sheaves on $\Z$ with additional structure'' is less straightforward.
\end{rem}

We will use the phrase $\cO_{\Z^\prime}$-module to denote a quasicoherent sheaf on the stack $\Z^\prime = Z/P$. $\lambda$ fixes $Z$, so $P$-equivariant $\sh{O}_Z$-modules have a natural grading by the weight spaces of $\lambda$, and we will use this grading often.


\begin{lem}
For any $F \in \qcoh{\Z^\prime}$ and any $w \in \bZ$, the $\cO_Z$-submodule $F_{\geq w} := \sum_{i \geq w} F_i$ of sections of weight $\geq w$ with respect to $\lambda$ is $P$ equivariant.
\end{lem}
\begin{proof}
$\lambda(\Gm)$ commutes with $L$, so $F_{\geq w}$ is an equivariant submodule with respect to the $L$ action. Because $U \subset P$ acts trivially on $Z$, the $U$-equivariant structure on $F$ is determined by a coaction $a : F \to k[U] \otimes F$, which is equivariant for the $\Gm$-action. We have
$$a(F_{\geq w}) \subset \left( k[U] \otimes F \right)_{\geq w} = \bigoplus_{i + j \geq w} k[U]_i \otimes F_j \subset k[U] \otimes F_{\geq w}$$
The last inclusion is due to the fact that $k[U]$ is non-positively graded, and it implies that $F_{\geq w}$ is equivariant with respect to the $U$ action as well. Because we have a semidirect product decomposition $P = U L$, it follows that $F_{\geq p}$ is an equivariant submodule with respect to the $P$ action.
\end{proof}

\begin{rem}
This lemma is a global version of the observation that for any $P$-module $M$, the subspace $M_{\geq w}$ with weights $\geq w$ with respect to $\lambda$ is a $P$-submodule, which can be seen from the coaction $M \to k[P] \otimes M$ and the fact that $k[P]$ is nonnegatively graded with respect to $\lambda$.
\end{rem}

It follows that any $F \in \qcoh{\Z^\prime}$ has a functorial factorization $F_{\geq w} \hookrightarrow F \twoheadrightarrow F_{<w}$. Note that as $\Gm$-equivariant instead of $P$-equivariant $\sh{O}_Z$-modules there is a natural isomorphism $F \simeq F_{\geq w} \oplus F_{<w}$. Thus the functors $(\bullet)_{\geq w}$ and $(\bullet)_{<w}$ are exact, and if $F$ is locally free, then $F_{\geq w}$ and $F_{<w}$ are locally free as well.

We define $\qcoh{\Z^\prime}_{\geq w}$ and $\qcoh{\Z^\prime}_{<w}$ to be the full subcategories of $\qcoh{\Z^\prime}$ consisting of sheaves supported in weight $\geq w$ and weight $<w$ respectively. They are both Serre subcategories, they are orthogonal to one another, $(\bullet)_{\geq w}$ is right adjoint to the inclusion $\qcoh{\Z^\prime}_{\geq w} \subset \qcoh{\Z^\prime}$, and $(\bullet)_{<w}$ is left adjoint to the inclusion $\qcoh{\Z^\prime}_{<w} \subset \qcoh{\Z^\prime}$.

\begin{lem} \label{lem_special_resolutions}
Any $F \in \qcoh{\Z^\prime}_{<w}$ admits an injective resolution $F \to \sh{I}^0 \to \sh{I}^1 \to \cdots$ such that $\sh{I}^i \in \qcoh{\Z^\prime}_{<w}$. Likewise any $F \in \coh{\Z^\prime}_{\geq w}$ admits a locally free resolution $\cdots \to E_1 \to E_0 \to F$ such that $E_i \in \coh{\Z^\prime}_{\geq w}$.
\end{lem}
\begin{proof}
First assume $F \in \qcoh{\Z^\prime}_{< w}$, and let $F \to \sh{I}^0$ be the injective hull of $F$ in $\qcoh{\Z^\prime}$.\footnote{The injective hull exists because $\qcoh{\Z^\prime}$ is cocomplete and taking filtered colimits is exact.} Then $\sh{I}^0_{\geq w} \cap F_{<w} = 0$, hence $\sh{I}^0_{\geq w} = 0$ because $\sh{I}^0$ is an essential extension of $F$. $\qcoh{\Z^\prime}_{<w}$ is a Serre subcategory, so $\sh{I}^0 / F \in \qcoh{\Z^\prime}_{<w}$ as well, and we can inductively build an injective resolution with $\sh{I}^i \in \qcoh{\Z^\prime}_{<w}$.

Next assume $F \in \coh{\Z^\prime}_{\geq w}$. Choose a surjection $E \to F$ where $E$ is locally free. Then $E_0 := E_{\geq w}$ is still locally free, and $E_{\geq w} \to F$ is still surjective. Because $\coh{\Z^\prime}_{\geq w}$ is a Serre subcategory, $\ker (E_0 \to F) \in \coh{\Z^\prime}_{\geq w}$ as well, so we can inductively build a locally free resolution with $E_i \in \coh{\Z^\prime}_{\geq w}$.
\end{proof}

We will use this lemma to study the subcategories of $\D^b(\Z^\prime)$ generated by $\coh{\Z^\prime}_{\geq w}$ and $\coh{\Z^\prime}_{<w}$. Define the full triangulated subcategories, where $?$ can denote either $-,+,b,$ or blank.
\begin{gather*}
\D^?(\Z^\prime)_{\geq w} = \{ F^\bdot \in \D^?(\Z^\prime) | \cH^i(F^\bdot) \in \qcoh{\Z^\prime}_{\geq w} \text{ for all } i \} \\
\D^?(\Z^\prime)_{< w} = \{ F^\bdot \in \D^?(\Z^\prime) | \cH^i(F^\bdot) \in \qcoh{\Z^\prime}_{< w} \text{ for all }i\}
\end{gather*}

For any complex $F^\bdot$ we have the canonical short exact sequence
\begin{equation} \label{eqn_canonical_sequence_Z}
0 \to F^\bdot_{\geq w} \to F^\bdot \to F^\bdot_{<w} \to 0
\end{equation}
If $F^\bdot \in \D^b(\Z^\prime)_{\geq w}$ then the first arrow is a quasi-isomorphism, because $(\bullet)_{\geq w}$ is exact. Likewise for the second arrow if $F^\bdot \in \D^b(\Z^\prime)_{<w}$. Thus $F^\bdot \in \D^b(\Z^\prime)_{\geq w}$ iff it is quasi-isomorphic to a complex of sheaves in $\coh{\Z^\prime}_{\geq w}$ and likewise for $\D^b(\Z^\prime)_{< w}$.

\begin{prop} \label{prop_baric_Z}
These subcategories constitute a \emph{baric decomposition}
$$\D^b(\Z^\prime) = \langle \D^b(\Z^\prime)_{< w}, \D^b(\Z^\prime)_{\geq w} \rangle$$
This baric decomposition is \emph{multiplicative} in the sense that
$$\op{Perf}(\Z^\prime)_{\geq w} \otimes \D^b(\Z^\prime)_{\geq v} \subset \D^b(\Z^\prime)_{\geq v+w}.$$
It is \emph{bounded}, meaning that every object lies in $\sh{D}_{\geq w} \cap \sh{D}_{< v}$ for some $w,v$. The baric truncation functors, the adjoints of the inclusions $\sh{D}_{\geq w},\sh{D}_{<w} \subset \D^b(\Z^\prime)$, are exact.
\end{prop}
\begin{proof}
If $A \in \coh{\Z^\prime}_{\geq w}$ and $B \in \coh{\Z^\prime}_{<w}$, then by Lemma \ref{lem_special_resolutions} we resolve $B$ by injectives in $\qcoh{\Z^\prime}_{<w}$, and thus $R\op{Hom}(A,B) \simeq 0$. It follows that $\D^b(\Z^\prime)_{\geq w}$ is left orthogonal to $\D^b(\Z^\prime)_{<w}$. $\qcoh{\Z^\prime}_{\geq w}$ and $\qcoh{\Z^\prime}_{< w}$ are Serre subcategories, so $F^\bdot_{\geq w} \in \D^b(\Z^\prime)_{\geq w}$ and $F^\bdot_{<w} \in \D^b(\Z^\prime)_{<w}$ for any $F^\bdot \in \D^b(\Z^\prime)$. Thus the natural sequence \eqref{eqn_canonical_sequence_Z} shows that we have a baric decomposition, and that the right and left truncation functors are the exact functors $(\bullet)_{\geq w}$ and $(\bullet)_{<w}$ respectively. Boundedness follows from the fact that coherent equivariant $\sh{O}_Z$-modules must be supported in finitely many $\lambda$ weights. Multiplicativity is also straightforward to verify.
\end{proof}

\begin{rem}
A completely analogous baric decomposition holds for $\Z$ as well. In fact, for $\Z$ the two factors are mutually orthogonal.
\end{rem}

\begin{center} $\star \star \star$ \end{center}
\vskip 5 pt

Next we turn to the derived category of $\S$. The closed immersion $\sigma : Z \hookrightarrow Y$ is $L$ equivariant, hence it defines a map of stacks $\sigma : \Z \to \S$. Recall also that because $\pi : \S \to \Z^\prime$ is affine, the derived pushforward $R\pi_\ast = \pi_\ast$ is just the functor which forgets the $\cA$-module structure. Define the thick triangulated subcategories
\begin{equation*}
\begin{array}{ll}
\D^?(\S)_{<w} = \{ F^\bdot \in \D^?(\S) | \pi_\ast F^\bdot \in \D(\Z^\prime)_{<w} \}, & ? = -,+,b, \text{ or blank} \\[8pt]
\D^?(\S)_{\geq w} = \{ F^\bdot \in \D^?(\S) | L\sigma^\ast F^\bdot \in \D^-(\Z)_{\geq w} \}, & ? = -,b
\end{array}
\end{equation*}
In the rest of this subsection we will analyze these two categories and show that they constitute a multiplicative baric decomposition.

\begin{prop} \label{prop_baric_S_flat}
Let $\S$ be a KN stratum such that $\pi : Y \to Z$ is flat. Then the categories $\D^b(\S) = \langle \D^b(\S)_{< w}, \D^b(\S)_{\geq w} \rangle$ constitute a multiplicative baric decomposition, and the truncation functors satisfy
$$\sigma^\ast (\radj{w} F^\bdot) \simeq (\sigma^\ast F^\bdot)_{\geq w} \text{ and } \sigma^\ast(\ladj{w} F^\bdot) \simeq (\sigma^\ast F^\bdot)_{<w}.$$
The baric decomposition restricts to a \emph{bounded} multiplicative baric decomposition of $\op{Perf}(\S)$, and if $Z \hookrightarrow Y$ has finite tor dimension then the baric decomposition on $\D^b(\S)$ is bounded as well.
\end{prop}

Note that when Property \hyperref[property_A]{(A)} holds, then $\pi : \S \to \Z$ is flat and $\sigma : Z \hookrightarrow Y$ has finite tor dimension. We will prove the proposition after collecting several key lemmas on the structure of the category $\D^-(\S)$.

\begin{lem}
Any object $F^\bdot \in \D^-(\S)$ admits a presentation as a right-bounded complex of the form $\cA \otimes E^\bdot$, where $E^i \in \coh{\Z^\prime}$ are locally free.
\end{lem}
\begin{proof}
By the standard method of constructing resolutions by vector bundles, it suffices to show that every object $F \in \coh{\S}$ admits a surjection from an object of the form $\cA \otimes E$ with $E \in \coh{\Z^\prime}$ locally free. Regarding $F$ as a coherent $\cA$-module on $Z$ and forgetting the $\cA$-module structure, $F$ is a union of its coherent $\cO_{\Z^\prime}$-submodules $F^\prime \subset F$ \cite{AB10}. Thus $F$ is a union of the coherent $\cA$-submodules $\cA \cdot F^\prime$, and because $F$ is coherent and $\cA$ is Noetherian, one must have $F = \cA \cdot F^\prime$ for some coherent $\cO_{\Z^\prime}$-submodule $F^\prime$.  Finally we choose a surjection $E \to F^\prime$ from a locally free sheaf on $\Z^\prime$, and it follows that $\cA \otimes E \to F$ is surjective.
\end{proof}

Complexes on $\S$ of the form $\sh{A} \otimes E^\bdot$, where each $E^i$ is a locally free sheaf on $\Z^\prime$, will be of prime importance. Note that the differential $d^i:\sh{A} \otimes E^i \to \sh{A} \otimes E^{i+1}$ is not necessarily induced from a differential $E^i \to E^{i+1}$. However we observe
\begin{lem} \label{lem_radj_of_pullback_S}
If $E \in \qcoh{\Z^\prime}$, then $\cA \cdot (\sh{A} \otimes E)_{\geq w} = \sh{A} \otimes E_{\geq w}$, where the left side denotes the smallest $\cA$-submodule containing the $\cO_Z$-submodule $(\cA\otimes E)_{\geq w}$.
\end{lem}
\begin{proof}
By definition the left-hand side is the $\sh{A}$-submodule generated by $\bigoplus_{i+j \geq w} \sh{A}_i \otimes E_j$ and the right-hand side is generated by $\bigoplus_{j \geq w} \sh{A}_0 \otimes E_j \subset \sh{A} \otimes E_{\geq w}$. These $\cO_Z$-submodules clearly generate the same $\cA$-submodule.
\end{proof}

This guarantees that $d^i(\cA \otimes E^i_{\geq w}) \subset \cA \otimes E^{i+1}_{\geq w}$, so $\sh{A} \otimes E^\bdot_{\geq w}$ is a subcomplex, and $E_{\geq w}$ is a direct summand as a non-equivariant $\sh{O}_Z$-module, so we have a canonical short exact sequence of complexes in $\qcoh\S$
\begin{equation} \label{eqn_canonical_sequence_S}
0 \to \sh{A} \otimes E^\bdot_{\geq w} \to \sh{A} \otimes E^\bdot \to \sh{A} \otimes E^\bdot_{<w} \to 0.
\end{equation}
We observe the following extension of Nakayama's lemma to the derived category
\begin{lem}[Nakayama] \label{lem:Nakayama}
If $F^\bdot \in \D^-(\S)$ and $L\sigma^\ast F^\bdot \simeq 0$, then $F^\bdot \simeq 0$.
\end{lem}
\begin{proof}
The natural extension of Nakayama's lemma to stacks is the statement that the support of a coherent sheaf is closed. In our setting this means that if $G \in \coh\S$ and $G \otimes \sh{O}_Z = 0$ then $G=0$, because $\op{supp}(G) \cap Z = \emptyset$ and every nonempty closed substack of $S$ intersects $Z$ nontrivially. 

If $H^r (F^\bdot)$ is the highest nonvanishing cohomology group of a right bounded complex, then $H^r (L\sigma^\ast F^\bdot) \simeq \sigma^\ast H^r(F^\bdot)$. By Nakayama's lemma $\sigma^\ast H^r(F^\bdot) = 0 \Rightarrow H^r(F^\bdot) = 0$, so we must have $\sigma^\ast H^r(F^\bdot) \neq 0$ as well.
\end{proof}

With this we can explicitly characterize the category $\D^-(\S)_{\geq w}$:

\begin{lem} \label{lem_characterize_radj_S}
$F^\bdot \in \D^-(\S)_{\geq w}$ iff it is quasi-isomorphic to a right-bounded complex of sheaves of the form $\sh{A} \otimes E^i$ with $E^i \in \coh{\Z^\prime}_{\geq w}$ locally free.
\end{lem}

\begin{proof}
We assume that $L\sigma^\ast F^\bdot \in \D^-(\Z)_{\geq w}$. Choose a right-bounded presentation by locally frees $\sh{A} \otimes E^\bdot \simeq F^\bdot$ and consider the canonical sequence \eqref{eqn_canonical_sequence_S}.

Restricting to $\Z$ gives a short exact sequence $0 \to E^\bdot_{\geq w} \to E^\bdot \to E^\bdot_{<w} \to 0$. The first and second terms have homology in $\coh\Z_{\geq w}$, and the third has homology in $\coh\Z_{<w}$. These two categories are orthogonal, so it follows from the long exact homology sequence that $E^\bdot_{<w}$ is acyclic. Thus by Nakayama's lemma $\sh{A} \otimes E^\bdot_{<w}$ is acyclic and $F^\bdot \simeq \sh{A} \otimes E^\bdot_{\geq w}$.
\end{proof}

Using this characterization of $\D^-(\S)_{\geq w}$ we have semiorthogonality:
\begin{lem} \label{lem_left_orthog_S}
For all $F^\bdot \in \D^-(\S)_{\geq w}$ and $G^\bdot \in \D^+_{qc}(\S)_{<w}$, we have $R \op{Hom}_\S (F^\bdot, G^\bdot) = 0$.
\end{lem}
\begin{proof}
By Lemma \ref{lem_characterize_radj_S} if suffices to prove the claim for $F^\bdot = \cA \otimes E$ with $E \in \coh{\Z^\prime}_{\geq w}$ locally free. Then $\sh{A} \otimes E \simeq L\pi^\ast E$, and the derived adjunction gives $R\op{Hom}_\S(L\pi^\ast E,G^\bdot) \simeq R\op{Hom}_{\Z^\prime}(E,R\pi_\ast G^\bdot)$. $\pi$ is affine, so $R\pi_\ast G^\bdot \simeq \pi_\ast G^\bdot \in \D^+_{qc}(\Z^\prime)_{<w}$. The claim follows from the fact that $\qcoh{\Z^\prime}_{\geq w}$ is left orthogonal to $\D^+_{qc}(\Z^\prime)_{<w}$.
\end{proof}

\begin{rem}
The category of coherent $\cO_\S$-modules whose weights are $<w$ is a Serre subcategory of $\coh\S$ generating $\D^b(\S)_{<w}$, but there is no analogue for $\D^b(\S)_{\geq w}$. For instance, when $G$ is abelian there is a short exact sequence $0 \to \sh{A}_{<0} \to \sh{A} \to \sh{O}_Z \to 0$. This nontrivial extension shows that $R\op{Hom}_\S(\sh{O}_Z, \sh{A}_{<0}) \neq 0$ even though $\cO_Z$ has nonnegative weights.
\end{rem}

Every $F \in \coh\S$ has a highest weight subsheaf, $F_{\geq h}$, when regarded as an equivariant $\sh{O}_Z$-module, where by definition $F_{\geq h} \neq 0$ but $F_{\geq w} = 0$ for $w > h$. Furthermore, because $\sh{A}_{<0}$ has strictly negative weights the map $(F)_{\geq h} \to (F\otimes \sh{O}_Z)_{\geq h}$ is an isomorphism of $L$-equivariant $\cO_Z$-modules. Using the notion of highest weight submodule we prove

\begin{lem} \label{lem_truncation_S}

Let $\cA \otimes E^\bdot \in \D^b(\S)$ be a right-bounded complex as above. Assume that either of the following holds:
\begin{itemize}
\item $\Z$ is smooth;
\item $\pi : \S \to \Z^\prime$ is flat.
\end{itemize}
Then $\cA \otimes E^\bdot_{\geq w}$ and $\cA \otimes E^\bdot_{<w}$ have bounded cohomology. If $\cA \otimes E^\bdot$ is perfect, then so are $\cA \otimes E^\bdot_{\geq w}$ and $\cA \otimes E^\bdot_{< w}$.
\end{lem}

\begin{proof}

First we show that for $W$ sufficiently large, $\cA \otimes E^\bdot_{\geq W} \simeq 0$. By Nakayama's lemma and the fact that $\cA \otimes E^\bdot$ has bounded cohomology, it suffices to show $(L\sigma^\ast F)_{\geq W} \simeq 0$ for any $F \in \coh\S$, and this follows by constructing a resolution of $F$ by vector bundles whose weights are $\leq$ the highest weight of $F$.

Now assume that $\cA \otimes E^\bdot_{\geq w+1} \in \D^b(\S)$. It follows from the sequence \eqref{eqn_canonical_sequence_S} that $\cA \otimes E^\bdot_{<w+1}$ has bounded cohomology. Applying \eqref{eqn_canonical_sequence_S} to the complex $\cA \otimes E^\bdot_{\geq w}$ gives
$$0 \to \cA \otimes E^\bdot_{\geq w+1} \to \cA \otimes E^\bdot_{\geq w} \to \cA \otimes E^\bdot_{w} \to 0,$$
where $E^\bdot_{w}$ denotes the subquotient of $E^\bdot$ concentrated in weight $w$. In order to show that $\cA \otimes E^\bdot_{\geq w} \in \D^b(\S)$, it suffices to show that $\cA \otimes E^\bdot_{w} \in \D^b(\S)$.

The differential for the complex $\cA \otimes E^\bdot_w$ is uniquely determined by the $P$-equivariant maps of $\cO_Z$-modules $d^i : E^i_w \to \cA \otimes E^{i+1}_w$, which must factor through $E^{i+1}_w \subset \cA \otimes E^{i+1}_w$ because that is precisely the subsheaf of weight $w$. It follows that $E_w^\bdot$, the highest weight subsheaves of $\cA \otimes E^\bdot_{<w+1}$, are a subcomplex as $P$-equivariant $\cO_Z$-modules, and $\cA \otimes E^\bdot_w = L\pi^\ast (E^\bdot_w)$. Furthermore $E^\bdot_w$ is a summand of $\cA \otimes E^\bdot$ as non-equivariant $\cO_Z$-modules and thus has bounded cohomology. If we assume that $\pi : \S \to \Z^\prime$ is flat, then it follows that $L\pi^\ast (E^\bdot_w) \in \D^b(\S)$ as well.

Under the hypothesis that $\Z$ is smooth, we modify the previous argument slightly. If $\Z$ is smooth then so is $\Z^\prime$, so the complex $E^\bdot_w$ is actually in $\op{Perf}(\Z^\prime)$. It follows that $\cA \otimes E^\bdot_w = L\pi^\ast (E^\bdot_w)$ is perfect as well, and in particular has bounded cohomology.

If $\cA \otimes E^\bdot$ is a right bounded complex, then $E^\bdot = \cO_Z \otimes_{\cA} (\cA \otimes E^\bdot)$, regarded as a complex of $L$-equivariant $\cO_Z$-modules, is precisely $\sigma^\ast (\cA \otimes E^\bdot)$. The final claim in the lemma regarding perfect complexes follows from two observations: 1) the baric truncation functors on $\D^b(\Z)$ preserve perfect complexes, and 2) $F^\bdot \in \D^b(\S)$ is perfect if and only if $\sigma^\ast F^\bdot$ is perfect. To prove the second claim, note that an object in $\D^-(\S)$ is perfect if and only if its pullback to $Y/L$ is perfect. Furthermore, if $F^\bdot \in \D^b(Y/L)$ and $F^\bdot|_{Z}$ is perfect, then $F^\bdot$ is perfect in an $L$-equivariant open neighborhood of $Z$, and the only such open subset is all of $Y$.

 
\end{proof}

\begin{proof}[Proof of Proposition \ref{prop_baric_S_flat}]
Lemma \ref{lem_left_orthog_S} implies $\D^b(\S)_{\geq w}$ is left orthogonal to $G^\bdot \in \D^b(\S)_{<w}$. In order to obtain baric truncations for $F^\bdot \in \D^b(\S)$, we choose a presentation of the form $\cA \otimes E^\bdot$ with $E^\bdot \in \coh{\Z^\prime}$ locally free. The canonical short exact sequence \eqref{eqn_canonical_sequence_S} gives an exact triangle $\sh{A}\otimes E^\bdot_{\geq w} \to F^\bdot \to \sh{A}\otimes E^\bdot_{<w} \parr$. By Lemma \ref{lem_truncation_S} all three terms have bounded cohomology, thus our truncations are $\radj{w} F^\bdot = \cA \otimes E^\bdot_{\geq w}$ and $\ladj{w} F^\bdot = \cA \otimes E^\bdot_{< w}$. The claim about the restrictions $\sigma^\ast \radj{w}$ and $\sigma^\ast \ladj{w}$ follows from the observation that applying the functor $\sigma^\ast \simeq \cO_Z \otimes_\cA (\bullet)$ to the sequence \eqref{eqn_canonical_sequence_S} gives the sequence \eqref{eqn_canonical_sequence_Z} for $\sigma^\ast F^\bdot$.

If $F^\bdot \in \op{Perf}(\S)$, then by Lemma \ref{lem_truncation_S} so are $\radj{w} F^\bdot$ and $\ladj{w}F^\bdot$. The multiplicativity of $\D^b(\S)_{\geq w}$ follows from the fact that $\D^b(\Z)_{\geq w}$ is multiplicative and the fact that $L\sigma^\ast$ respects derived tensor products.

Every $M \in \coh\S$ has a highest weight space, so $M \in \D^b(\S)_{<w}$ for some $w$. This implies that any $F^\bdot \in \D^b(\S)$ lies in $\D^b(\S)_{<w}$ for some $w$. The analogous statement for $\D^b(\S)_{\geq w}$ is false in general, but if $F^\bdot \in \D^b(\S)$ is such that $\sigma^\ast F^\bdot$ is cohomologically bounded, then $F^\bdot \in \D^b(\S)_{\geq w}$ for some $w$. The boundedness properties follow from this observation.
\end{proof}


Combining the fact that $\sigma^\ast \radj{w} \simeq (\sigma^\ast(\bullet))_{\geq w}$ and $\sigma^\ast \ladj{w} \simeq (\sigma^\ast(\bullet))_{< w}$ with Nakayama's lemma \ref{lem:Nakayama} gives the following:
\begin{cor} \label{cor_characterize_baric_S}
Let $\S$ be a KN stratum such that $\pi : Y \to Z$ is flat, and let $F^\bdot \in \D^b(\S)$. Then $F^\bdot \in \D^b(\S)_{\geq w}$ if and only if $\sigma^\ast(F^\bdot) \in \D^-(\Z)_{\geq w}$, and $F^\bdot \in \D^b(\S)_{<w}$ if and only if $\sigma^\ast \D^-(\Z)_{<w}$.
\end{cor}

It is useful to have a more flexible method of computing the truncations $\radj{w}$ and $\ladj{w}$, as well as a more explicit description of the category
$$\D^b(\S)_w := \D^b(\S)_{\geq w} \cap \D^b(\S)_{<w+1}.$$

\begin{amplif} \label{amplif_KN_stratum}
Let $\S$ be a KN stratum such that $\pi : Y \to Z$ is flat. For any $F^\bdot \in \D^b(\S)$, $\radj{w} F^\bdot$ and $\ladj{w} F^\bdot$ can be computed from a presentation $F^\bdot \simeq \cA \otimes E^\bdot$ with $E^i \in \coh{\Z^\prime}$ coherent but not necessarily locally free. Furthermore regarding $\pi$ as the morphism $\S \to \Z$, the pullback functor $\pi^\ast : \D^b(\Z)_w \to \D^b(\S)_w$ is an equivalence. Thus if $\sigma : Z \hookrightarrow Y$ has finite tor dimension, we have an infinite semiorthogonal decomposition
$$\D^b(\S) = \langle \ldots, \D^b(\Z)_w, \D^b(\Z)_{w+1}, \ldots \rangle.$$
\end{amplif}

\begin{proof}
For $E \in \coh{\Z^\prime}$, Corollary \ref{cor_characterize_baric_S} implies that $\cA \otimes E \in \D^b(\S)_{\geq w}$ if and only if $E \in \coh{\Z^\prime}_{\geq w}$, and likewise for $<w$. Thus
$$\cA \otimes E^\bdot_{\geq w} \to F^\bdot \to \cA \otimes E^\bdot_{<w} \parr$$
is an exact triangle exhibiting $F^\bdot$ as an extension of an object in $\D^-(\S)_{<w}$ by an object of $\D^-(\S)_{\geq w}$. It follows that the morphism $\radj{w} F^\bdot \to F^\bdot$ factors uniquely through $\cA \otimes E^\bdot_{\geq w}$. Furthermore $\sigma^\ast \radj{w} F^\bdot \to \sigma^\ast (\cA \otimes E^\bdot_{\geq w})$ is an equivalence, so it it follows from Lemma \ref{lem:Nakayama} that $\radj{w} F^\bdot \to \cA \otimes E^\bdot_{\geq w}$ is an equivalence and hence $\ladj{w} F^\bdot \simeq \cA \otimes E^\bdot_{< w}$.

In fact for any coherent $\cA$-module $M$ there is a coherent $E \in \coh{\Z^\prime}$ and a surjection $\cA \otimes E \twoheadrightarrow M$ which is an isomorphism on highest weight subsheaves, and one can use this fact to construct a presentation of this form in which $E^i_{\geq w} = 0$ for $i\ll 0$. So in fact $\radj{w} F^\bdot$ is equivalent to a \emph{finite} complex of the form $\cA \otimes E^\bdot_{\geq w}$.

Regarding $\sigma$ as a morphism $\Z \to \S$, the functor $(\sigma^\ast(\bullet))_{\geq w}$ is the inverse to $\pi^\ast : \D^b(\Z)_w \to \D^b(\S)_w$. Because $\sigma$ is a section of $\pi$, we have a canonical map $(\sigma^\ast \pi^\ast (E^\bdot))_{\geq w} \to E^\bdot$ which is an isomorphism for $E^\bdot \in \D^b(\Z)_w$. Furthermore for any object of the form $F = \cA \otimes E$ with $E \in \coh\Z_w$, we have
$$\pi^\ast ( (\sigma^\ast F^\bdot)_{\geq w}) \simeq \pi^\ast E \simeq F^\bdot.$$
It follows that $\pi^\ast$ is fully faithful on $\coh\Z_w$. By the previous paragraph, objects of the form $\cA \otimes E$ with $E \in \coh\Z_w$ generate $\D^b(\S)_w$ as a triangulated category, so $\pi^\ast$ is an equivalence. The existence of an infinite semiorthogonal decomposition follows formally from the existence of a bounded baric decomposition.
\end{proof}

\begin{rem} \label{rem_extending_SOD}
The baric decomposition of Proposition \ref{prop_baric_S_flat} extends uniquely to a baric decomposition
$$\D_{qc}(\S) = \langle \D_{qc}(\S)_{<w}, \D_{qc}(\S)_{\geq w} \rangle$$
The category $\D_{qc}(\S)_{<w}$ can still be described as those $F^\bdot \in \D_{qc}(\S)$ such that $\cH^p(F^\bdot)$ has weight $<w$ for all $i$. However, there are objects $F^\bdot \neq 0$ for which $\sigma^\ast F^\bdot = 0$, so $\sigma^\ast$ can no longer be used to characterize $\D_{qc}(\S)_{\geq w}$ and $\D_{qc}(\S)_{<w}$.
\end{rem}

The baric decomposition of $\D_{qc}(\S)$ follows from the fact that the stack $\S$ is perfect in the sense of \cite{BFN10}, so we can apply a general fact about compactly generated categories.

\begin{lem} \label{lem_extending_SOD_general}
Let $\cT$ be a cocomplete triangulated category which is the homotopy category of a pretriangulated dg-category, and let $\cC \subset \cT^c$ be a triangulated subcategory of compact objects which generates $\cT$. If $\cC = \langle \cA, \cB \rangle$, then $\cT = \langle \bar{\cA}, \bar{\cB} \rangle$, where $\bar{\cA}$ (resp. $\bar{\cB}$) denotes the smallest cocomplete triangulated subcategory containing $\cA$ (resp. $\cB$). The projection functors onto $\bar{\cA}$ and $\bar{\cB}$ commute with colimits.
\end{lem}

\begin{proof}
By Brown-Neeman representability $\bar{\cB}$ is right admissible, so we have $\cT = \langle \bar{\cB}^\perp, \bar{\cB} \rangle$. $\cA$ lies in $\bar{\cB}^\perp$ because $\cA \subset \cT^c$, and it generates because $\cA$ and $\cB$ together generate $\cT$.

If we let $M = \colim M_\alpha$, then by the functoriality of homotopy colimits we can form the exact triangle
$$\varinjlim i^R_{\bar{\cB}}M_\alpha \to M \to \varinjlim i^L_{\bar{\cA}} M_\alpha,$$
where $i^R_{\bar{\cB}}$ and $i^L_{\bar{\cA}}$ denote the right and left adjoints of the respective inclusions of subcategories. The first term lies in $\bar{\cB}$ and the third term lies in $\bar{\cA}$, so using the semiorthogonal decomposition of $\cT$ we have canonical isomorphisms $i^L_{\bar{\cA}}M \simeq \varinjlim i^L_{\bar{\cA}}M_\alpha$ and likewise for $i^R_{\bar{\cB}} M$.
\end{proof}

\subsection{The cotangent complex and Property \hyperref[property_L_plus]{(L+)}}
\label{sect_cotangent_complex}

We review the construction of the cotangent complex and prove the main implication of the positivity Property \hyperref[property_L_plus]{(L+)}:
\begin{lem} \label{lem_positivity}
If $\S \hookrightarrow \X$ satisfies Property \hyperref[property_L_plus]{(L+)} and $F^\bdot \in \D^b(\S)_{\geq w}$, then $Lj^\ast j_\ast F^\bdot \in \D^-(\S)_{\geq w}$ as well.
\end{lem}

We can inductively construct a cofibrant replacement $\cO_\S$ as an $\cO_\X$ module: a surjective weak equivalence $\varphi : \cB^\bdot \twoheadrightarrow \cO_\S$ from a sheaf of commutative dg-$\cO_\X$ algebras such that $\cB^\bdot \simeq ( \mathbb{S} (E^\bdot), d)$, where $\mathbb{S}(E^\bdot)$ is the free graded commutative sheaf of algebras on a graded sheaf of $\cO_\X$-modules, $E^\bdot$, with $E^i$ locally free and $E^i = 0$ for $i \geq 0$. Note that the differential is uniquely determined by its restriction to $E^\bdot$, and letting $e$ be a local section of $E^\bdot$ we decompose $d_\cB(e) = d_{\cB,-1} (e) + d_{\cB,0} (e) + \cdots$ where $d_{\cB,i}(e) \in \mathbb{S}^{i+1}(E^\bdot)$.

We let $\D_{qc}(\cB^\bdot)$ and $\D_{qc}(\cO_\S)$ denote the homotopy categories of sheaves of quasicoherent dg-modules over the respective sheaves of commutative dg-$\cO_\X$-algebras. Because $\varphi$ is a quasi-iso\-morphism, the pair of adjoint functors
\begin{equation} \label{eqn_morita_equivalence}
\xymatrix{ (\bullet) \otimes_{\cB^\bdot} \cO_\S : \D_{qc}(\cB^\bdot ) \ar@<.5ex>[r] &  \D_{qc}(\cO_\S) : \varphi_\ast \ar@<.5ex>[l] }
\end{equation}
are mutually inverse equivalences of categories, where $\varphi_\ast$ denotes the functor which simply regards a dg-$\cO_\S$-module as a dg-$\cB^\bdot$-module via $\varphi$. It is evident that this equivalence restricts to an equivalence of the full subcategories $\D^-$ consisting of complexes with coherent cohomology sheaves, vanishing in sufficiently high cohomological degree.

The $\cB^\bdot$-module of K\"{a}hler differentials is
$$\cB^\bdot \xrightarrow{\delta} \Omega^1_{\cB^\bdot / \cO_\X} = \mathbb{S}(E^\bdot) \otimes_{\cO_\X} E^\bdot$$
with the universal closed degree 0 derivation over $\cO_\X$ defined by $\delta(e) = 1 \otimes e$ and extended to all of $\cB^\bdot$ by the Leibniz rule. The differential on $\Omega^1_{\cB^\bdot / \cO_\X}$ is uniquely determined by its commutation with $\delta$
$$d(1 \otimes e) = \delta (de) = 1 \otimes d_{\cB,0}(e) + \delta (d_{\cB,2}(e) + d_{\cB,3}(e) + \cdots)$$
By definition
$$\bL^\bdot (\S \hookrightarrow \X) := \cO_\S \otimes_{\cB^\bdot} \Omega^1_{\cB^\bdot / \cO_\X} \simeq \cO_\S \otimes E^\bdot$$
where the differential is the restriction of $d_0$.

\begin{proof}[Proof of Lemma \ref{lem_positivity}]
For any $M^\bdot \in \D^-(\S)$, the equivalence \eqref{eqn_morita_equivalence} implies that we can find a unique (up to weak equivalence) $\cB^\bdot$-module of the form $\cB^\bdot \otimes_{\cO_\X} F^\bdot$ such that $\cO_S \otimes_{\cB^\bdot} (\cB^\bdot \otimes F^\bdot) \simeq M^\bdot$. Here we are assuming that each $F^i$ is a locally free $\cO_\X$ module and $F^i = 0$ for $i\gg 0$. The differential on $\cB^\bdot \otimes F^\bdot = \mathbb{S}(E^\bdot) \otimes F^\bdot$ is determined by the Leibniz rule and the homomorphism of $\cO_\X$-modules $d: F^\bdot \to \cB^\bdot \otimes F^\bdot$.

Locally we can choose a trivialization of $F^i$ given by sections $\{f_\alpha\}$ and trivializations of $F^j$ for $j > i$ given by sections $\{f_\beta\}$. We define the structure constants $g_\alpha^\beta \in \mathbb{S}(E^\bdot)$ by the formula
\begin{equation} \label{eqn_differential_formula}
d(1 \otimes f_\alpha) = \sum g_\alpha^\beta \otimes f_\beta,
\end{equation}
where $g_\alpha^\beta \in \cB^{i-j}$ for $f_\beta \in F^j$. From this we see that $\cO_\S \otimes_{\cB^\bdot} (\cB^\bdot \otimes F^\bdot)$ is precisely the complex $F^\bdot|_\S$, where the differential is $d(f_\alpha) = \sum \varphi( g_\alpha^\beta) f_\beta$. Thus up to quasi-isomorphism we may assume that $M^\bdot = F^\bdot|_\S$ with this differential.

We can regard $j_\ast (M^\bdot)$ as the complex $\cO_\S \otimes_{\cO_\X} F^\bdot$ with an appropriate differential, and the canonical map $\cB^\bdot \otimes F^\bdot \to \cO_\S \otimes F^\bdot$ is a right-bounded resolution of the latter by locally free $\cO_\X$-modules (this expresses the fact that the unit of adjunction of \eqref{eqn_morita_equivalence} is a quasi-isomorphism). It follows that
$$Lj^\ast j_\ast (M^\bdot) \simeq \cO_S \otimes_{\cO_\X} (\cB^\bdot \otimes F^\bdot) \simeq \mathbb{S}(E^\bdot) \otimes F^\bdot|_\S.$$
As before the differential is determined, via the Leibniz rule, from its restriction to $F^\bdot = \mathbb{S}^0(E^\bdot) \otimes F^\bdot$. Choosing local trivializations as above, the differential is given by the same formula \eqref{eqn_differential_formula}, where the right hand side is interpreted as an element of $\mathbb{S}(E^\bdot) \otimes F^\bdot|_\S$. We let $P^\bdot = \mathbb{S}(E^\bdot) \otimes F^\bdot|_\S \in \D^-(\S)$ denote this complex.

Note that $d_{\cB,-1} = 0$ after restricting to $\S$, so the differential on $P^\bdot$ satisfies
$$d(\mathbb{S}^i(E^\bdot) \otimes F^\bdot|_\S) \subset \bigoplus_{j \geq i} \mathbb{S}^j(E^\bdot) \otimes F^\bdot|_\S.$$
Thus if we let $T_z$ denote the endomorphism of $P^\bdot$ which scales $\mathbb{S}^i(E^\bdot)\otimes F^\bdot$ by $z^i$, the rescaled differential $d_z = T_z \circ d \circ T_{z^{-1}}$ is polynomial in $z$, and thus defines a family, $P^\bdot_z$, of complexes of $\cO_\S$-modules over $\bA^1$ which is trivial over $\bA^1 - \{0\}$, by which we mean an element of $\D^-(\S \times \bA^1)$ such that all fibers away from $0 \in \bA^1$ are isomorphic to $P^\bdot$.


In order to prove the lemma, we must show that if $M^\bdot|_\Z$ has $\lambda$-weights $\geq w$, then so does $P^\bdot|_\Z = P^\bdot_1|_\Z$. We observe that setting $z=0$, the differential on $P^\bdot_z$ is precisely the differential on the tensor product complex, so
$$P^\bdot_0 \simeq \mathbb{S}(\bL^\bdot_{\S/\X}) \otimes M^\bdot.$$
By hypothesis $\bL^\bdot_{\S/\X}|_\Z \to (\bL^\bdot_{\S/\X}|_\Z)_{\geq 0}$ is a weak equivalence, so $\mathbb{S}_\Z(\bL^\bdot_{\S/\X}|_\Z) \to \mathbb{S}_\Z((\bL^\bdot_{\S/\X}|_\Z)_{\geq 0})$ is a weak equivalence with a complex of locally frees generated in nonnegative weights. Thus assuming $(M^\bdot|_\Z)_{<w} = 0$ we have
$$(P^\bdot_0|_\Z)_{<w} \simeq (\mathbb{S}(\bL^\bdot_{\S/\X}|_\Z)_{\geq 0} \otimes (M^\bdot|_{\Z}))_{<w} \sim 0.$$
Because $P^\bdot_z|_\Z \in \D^-(\bA^1 \times \Z)$, semicontinuity implies that $(P^\bdot_z|_\Z)_{<w} = 0$ for all $z\in \bA^1$, and the lemma follows.

\end{proof}

\subsection{Koszul systems and cohomology with supports} \label{sect_koszul_systems}

We recall some properties of the right derived functor of the subsheaf with supports functor $R \inner{\op{\Gamma}}_\S(\bullet)$. It can be defined by the exact triangle $R \inner{\op{\Gamma}}_\S(F^\bdot) \to F^\bdot \to i_\ast (F^\bdot|_{\V}) \parr$, and it is the right adjoint of the inclusion $\D_{\S,qc}(\X) \subset \D_{qc}(\X)$. It is evident from this exact triangle that if $F^\bdot \in \D^b(\X)$, then $R\inner{\op{\Gamma}}_\S(F^\bdot)$ is still bounded, but no longer has coherent cohomology. On the other hand the formula
\begin{equation} \label{eqn_local_cohomology}
R\inner{\op{\Gamma}}_\S(F^\bdot) = \varinjlim \inner{\op{Hom}}(\cO_\X / \cI_\S^i , F^\bdot)
\end{equation}
shows that the subsheaf with supports is canonically a colimit of coherent complexes.

We will use a more general method of computing the subsheaf with supports similar to the Koszul complexes which can be used in the affine case.
\begin{lem} \label{lem_koszul_system}
Let $\X = X/G$, where $X$ is quasiprojective scheme with a linearizable action of an algebraic group $G$, and let $S \subset X$ be a KN stratum. Then there is a direct system $K_0^\bdot \to K_1^\bdot \to \cdots$ in $\op{Perf}(\X)^{[0,N]}$ along with compatible maps $K_i^\bdot \to \cO_\X$ such that
\begin{enumerate}
\item $\cH^\ast(K^\bdot_i)$ is supported on $\S$
\item $\varinjlim ( K^\bdot_i \otimes F^\bdot ) \to F^\bdot$ induces an isomorphism $\varinjlim (K^\bdot_i \otimes F^\bdot) \simeq R\inner{\op{\Gamma}}_\S(F^\bdot)$.
\item $\op{Cone}(K_i^\bdot \to K_{i+1}^\bdot)|_{\Z} \in \D^b(\Z)_{<w_i}$ where $w_i \to -\infty$ as $i \to \infty$.
\end{enumerate}
We will call such a direct system a \emph{Koszul system} for $\S \subset \X$
\end{lem}

\begin{proof}
First assume $\X$ is smooth in a neighborhood of $\S$. Then $\cO_\X / \cI_\S^i$ is perfect, so \eqref{eqn_local_cohomology} implies that the derived duals $K^\bdot_i = (\cO_\X / \cI_\S^i) ^\dual$ satisfy properties (1) and (2) with $K^\bdot_i \to \cO_\X$ the dual of the map $\cO_\X \to \cO_\X / \cI_\S^i$. We compute the mapping cone
$$\op{Cone}(K_i^\bdot \to K_{i+1}^\bdot) = \left( \cI_\S^i / \cI_\S^{i+1} \right)^\dual = \left( j_\ast ( \mathbb{S}^i(\coNorm_\S \X ) ) \right)^\dual$$
Where the last equality uses the smoothness of $\X$ and $\S$. Because Property \hyperref[property_L_plus]{(L+)} is automatic for smooth $\X$, it follows from Lemma \ref{lem_positivity} that $Lj^\ast j_\ast ( \mathbb{S}^i ( \coNorm_\S \X ) ) \in \D^b(\S)_{\geq i}$, hence $\op{Cone}(K_i^\bdot \to K_{i+1}^\bdot)$ has weights $\leq -i$, and the third property follows.

If $X$ is not smooth in a neighborhood of $S$, then by hypothesis we have a $G$-equivariant closed immersion $\phi : X \hookrightarrow X^\prime$ and closed KN stratum $S^\prime \subset X^\prime$ such that $S$ is a connected component of $S^\prime \cap X$ and $X^\prime$ is smooth in a neighborhood of $S^\prime$. Then we let $K_i^\bdot \in \op{Perf}(\X)$ be the restriction of $L\phi^\ast(\cO_{\X^\prime} / \cI_{\S^\prime}^i)^\dual$. These $K_i^\bdot$ still satisfy the third property. The canonical morphism $\varinjlim (K_i^\bdot \otimes F^\bdot) \to R \inner{\op{\Gamma}}_{\S^\prime \cap \X} F^\bdot$ is an isomorphism because its push forward to $\X^\prime$, $\varinjlim \phi_\ast (K_i^\bdot \otimes F^\bdot) \to \phi_\ast R \inner{\op{\Gamma}}_{\S^\prime \cap \X} (F^\bdot) = R \inner{\op{\Gamma}}_{\S^\prime} \phi_\ast F^\bdot$, is an isomorphism. Thus the $K_i^\bdot$ form a Koszul system for $\S^\prime \cap \X$. Because $\S$ is a connected component of $\S^\prime \cap \X$, the complexes $R\inner{\op{\Gamma}}_{\S} K_i^\bdot$ form a Koszul system for $\S$.
\end{proof}

We note an alternative definition of a Koszul system, which will be useful below
\begin{lem} \label{lem:alternate_Koszul_property}
Property (3) of a Koszul system is equivalent to the property that for all $w$,
$$\op{Cone}(K_i^\bdot \to \cO_\X)|_\Z \in \D^b(\Z)_{<w} \text{ for all } i \gg 0$$
\end{lem}
\begin{proof}
Let us denote $C_i^\bdot := \op{Cone}(K_i^\bdot \to \cO_\X)$. By the octahedral axiom we have an exact triangle
\begin{equation} \label{eqn_two_defs_Koszul}
C_i^\bdot [-1] \to C_{i+1}^\bdot [-1] \to \op{Cone}(K_i^\bdot \to K_{i+1}^\bdot) \parr
\end{equation}
So the property stated in this Lemma implies property (3) of the definition of a Koszul system.

Conversely, let $K_i^\bdot$ be a Koszul system for $\S \subset \X$. We have an exact triangle $\varinjlim K_i^\bdot \to \cO_\X \to \varinjlim C_i^\bdot \parr$, which in light of the isomorphism $R\inner{\op{\Gamma}}_\S \cO_\X \simeq \varinjlim K_i^\bdot$ implies that $\varinjlim C^\bdot_i \simeq i_\ast \cO_\V$. Combining this with Remark \ref{rem_extending_SOD} and Lemma \ref{lem_extending_SOD_general} we have
$$0 = \radj{w} j^\ast \varinjlim C_i^\bdot \simeq \varinjlim \radj{w} j^\ast C_i^\bdot.$$
On the other hand the exact triangle \eqref{eqn_two_defs_Koszul} shows that $\radj{w} j^\ast C_i^\bdot \to \radj{w} j^\ast C_{i+1}^\bdot$ is an isomorphism for $i\gg 0$. It follows that for $i \gg0$, $\radj{w} j^\ast C_i^\bdot = 0$ and hence $j^\ast C_i^\bdot \in \D^b(\S)_{<w}$.
\end{proof}

\subsection{Quasicoherent sheaves with support on $\S$, and the quantization theorem}
We turn to the derived category $\D^b_\S(\X)$ of coherent sheaves on $\X$ with set-theoretic support on $\S$. We will extend the baric decomposition of $\D^b(\S)$ to a baric decomposition of $\D^b_\S(\X)$. Using this baric decomposition we will prove a generalization of the quantization commutes with reduction theorem, one of the results which motivated this work.

\begin{defn}
We define the thick triangulated subcategories of $\D_{qc}(\X)$
\begin{equation*}
\begin{array}{ll}
\D^?(\X)_{\geq w} := \{ F^\bdot \in \D^?(\X) | Lj^\ast F^\bdot \in \D^-(\S)_{\geq w} \}, & ?=b \text{ or }- \\[8pt]
\D^?(\X)_{<w} := \{ F^\bdot \in \D^?(\X) | Rj^! F^\bdot \in \D^+(\S)_{<w} \}, & ? = b \text{ or } +
\end{array}
\end{equation*}
Futhermore we define $\D^?_\S(\X)_{\geq w} := \D_\S(\X) \cap \D^?(\X)_{\geq w}$ and $\D^?_\S(\X)_{<w} := \D_\S(\X) \cap \D^?(\X)_{<w}$.
\end{defn}

\begin{prop} \label{prop_baric_support_S}
Let $\S \subset \X$ be a KN stratum satisfying Properties \hyperref[property_A]{(A)} and \hyperref[property_L_plus]{(L+)}. There is a bounded multiplicative baric decomposition $\D^b_\S(\X) = \langle \D^b_\S(\X)_{<w}, \D^b_\S(\X)_{\geq w} \rangle$, which is compatible with the baric decomposition of $\D^b(\S)$ in the sense that $j_\ast \radj{w} \simeq \radj{w} j_\ast$ and $j_\ast \ladj{w} \simeq \ladj{w} j_\ast$.
\end{prop}
\begin{proof}
Let $\cA, \cB \subset \D^b(\X)$ be the full subcategories generated under cones and shifts by $j_\ast (\D^b(\S)_{\geq w})$ and $j_\ast (\D^b(\S)_{<w})$ respectively. By Lemma \ref{lem_positivity}, $Lj^\ast j_\ast (\D^b(\S)_{\geq w}) \subset \D^-(\S)_{\geq w}$, and so Lemma \ref{lem_left_orthog_S} implies that $\cB \subset \cA^\perp$. Consider the full subcategory $\sh{A} \star \sh{B} \subset \D^b_\S(\X)$ consisting of those $F^\bdot$ which admit triangles $A^\bdot \to F^\bdot \to B^\bdot \parr$ with $A^\bdot \in \sh{A}$ and $B^\bdot \in \sh{B}$. The right orthogonality $\cB \subset \cA^\perp$ implies that $\sh{A} \star \sh{B}$ is triangulated as well.

For any $F^\bdot \in \D^b(\S)$ we have the exact triangle $j_\ast \radj{w} F^\bdot \to j_\ast F^\bdot \to j_\ast \ladj{w} F^\bdot \parr$, so $j_\ast \D^b(\S) \subset \cA \star \cB$. But the smallest full triangulated subcategory containing $j_\ast \D^b(\S)$ is $\D^b_\S(\X)$, so we have a semiorthogonal decomposition $\D^b_\S(\X) = \langle \cB,\cA \rangle$. Finally, using the adjunctions $Lj^\ast \dashv j_\ast$ and $j_\ast \dashv Rj^!$ we can give alternate characterizations:
\begin{align*}
F^\bdot \in \cA &\Leftrightarrow R\op{Hom}^\bdot_\X (F^\bdot , j_\ast G^\bdot) = 0, \quad \forall G^\bdot \in \D^b(\S)_{<w} \\
&\Leftrightarrow R\op{Hom}^\bdot_\S (Lj^\ast F^\bdot, G^\bdot) = 0, \quad \forall G^\bdot \in \D^b(\S)_{<w} \\
&\Leftrightarrow Lj^\ast F^\bdot \in \D^-(\S)_{\geq w}
\end{align*}
A similar computation shows that $\cB = \D^b_\S(\X)_{<w}$.

The fact that the baric decomposition is multiplicative follows from the description of $\D^b_\S(\X)_{\geq w}$ in terms of the $\lambda$-weights of $L\sigma^\ast F^\bdot \in \D^-(\Z)$. Boundedness follows from the boundedness of the baric decomposition of Proposition \ref{prop_baric_S_flat} and the fact that $j_\ast \D^b(\S)$ generates $\D^b_\S(\X)$ under shifts and cones.
\end{proof}

\begin{rem}[Baric truncation functors are right $t$-exact] \label{rem_exact_truncation}
By construction the baric truncation functors on $\D^b(\S)$ are right $t$-exact. It follows that for $F^\bdot \in \D^b(\S)^{\leq 0}$, $$\radj{w} j_\ast F^\bdot := j_\ast \radj{w} F^\bdot \in \D^b_\S(\X)^{\leq 0},$$
and $\ladj{w} j_\ast F^\bdot \in \D^b_\S(\X)^{\leq 0}$ as well. Furthermore, $\D^b_\S(\X)^{\leq 0}$ is the smallest subcategory of $\D^b(\X)$ which contains $j_\ast ( \D^b(\S)^{\leq 0} )$ and is closed under extensions. It follows that $\radj{w}$ and $\ladj{w}$ are right $t$-exact on the category $\D^b_\S(\X)$.
\end{rem}

The following is an extension to our setting of an observation which appeared in \cite{BFK12}, following ideas of Kawamata \cite{Ka06}. There the authors described semiorthogonal factors appearing under VGIT in terms of the quotient $Z/L^\prime$.

\begin{amplif} \label{amplif_infinite_semidecomp}
Define $\D^b_\S(\X)_w := \D^b_\S(\X)_{\geq w} \cap \D^b_\S(\X)_{<w+1}$. If the weights of $\bL^\bdot_{S/X}|_Z$ are \emph{strictly positive}, then $j_\ast : \D^b(\S)_w \to \D^b_\S(\X)_w$ is an equivalence with inverse $\ladj{w+1} Lj^\ast (F^\bdot)$.
\end{amplif}

\begin{proof}
This is a consequence of the proof of Lemma \ref{lem_positivity}, which can be used to show that for $F\in \D^b(\S)_w$ the cone of the canonical morphism $Lj^\ast j_\ast F^\bdot \to F^\bdot$ lies in $\D^-(\S)_{\geq w+1}$. In that proof we showed that $Lj^\ast j_\ast F^\bdot$ is a deformation of $\mathbb{S}(\bL^\bdot_{\S/\X}) \otimes F^\bdot$. Using the same construction one can check that the counit of adjunction, $Lj^\ast j_\ast M \to M$, is a deformation of the augmentation map, $\mathbb{S}(\bL^\bdot_{\S / \X}) \otimes F^\bdot \to \cO_\S \otimes_{\cO_\S} F^\bdot = F^\bdot$. By hypothesis, $\op{Cone}(\mathbb{S}(\bL^\bdot_{\S / \X}) \to \cO_\S) \in \D^-(\S)_{\geq 1}$, so $\op{Cone}(\mathbb{S}(\bL^\bdot_{\S / \X}) \otimes F^\bdot \to F^\bdot) \in \D^-(\S)_{\geq w+1}$ and the claim follows from semicontinuity as in the proof of Lemma \ref{lem_positivity}.
\end{proof}

\begin{cor} \label{cor_infinite_semidecomp}
If $\bL^\bdot_{S/X}|_Z$ has strictly positive weights, then the baric decomposition of Proposition \ref{prop_baric_support_S} can be refined to an infinite semiorthogonal decomposition
$$\D^b_\S(\X) = \langle \ldots, \D^b(\Z)_{w-1},\D^b(\Z)_{w}, \D^b(\Z)_{w+1},\D^b(\Z)_{w+2}, \ldots \rangle$$
where the factors are the essential images of the fully faithful embeddings $j_\ast \pi^\ast : \D^b(\Z)_w \to \D^b_\S(\X)$.
\end{cor}

\medskip

Next we will use the baric decomposition of Proposition \ref{prop_baric_support_S} to generalize a theorem of Teleman \cite{Te00}, which (for smooth $X$) identifies a weight condition on an equivariant vector bundle $\cV$ which implies that $H^i(X,\cV)^G \simeq H^i(X^{ss},\cV)^G$.

\begin{thm}[Quantization Theorem]\label{thm_quantization}
Let $\S \subset \X$ be a KN stratum satisfying Properties \hyperref[property_A]{(A)} and \hyperref[property_L_plus]{(L+)}. Let $F^\bdot \in \D^-(\X)_{\geq w}$ and $G^\bdot \in \D^+(\X)_{<v}$ with $w \geq v$, then the restriction map
$$R \op{Hom}_\X (F^\bdot, G^\bdot) \to R \op{Hom}_\V (F^\bdot|_\V, G^\bdot|_\V)$$
to the open substack $\V = \X \setminus \S$ is an isomorphism.
\end{thm}

First we observe that the $t$-structure on $\D^+_\S(\X)$ preserves the subcategory $\D_\S^+(\X)_{<w}$.

\begin{lem} \label{lem_homology_weights}
Let $F^\bdot \in \D_\S^+(\X)$, then the following are equivalent:
\begin{enumerate}
\item $F^\bdot \in \D^+_\S(\X)_{<w}$,
\item $\tau^{\leq m} F^\bdot \in \D^b_\S(\X)_{<w}$ for all $m$, and
\item $\cH^m(F^\bdot) \in \D^b_\S(\X)_{<w}$ for all $m$.
\end{enumerate}
\end{lem}
\begin{proof}
It is clear that $(3) \Rightarrow (2) \Rightarrow (1)$, so we must show that $(1) \Rightarrow (3)$. If $\cH^p(F^\bdot)$ is the lowest non-vanishing homology sheaf, then we have an exact triangle
$$\cH^p(F^\bdot)[-p] \to F^\bdot \to \tau^{>p} F^\bdot \parr,$$
so by an inductive argument it suffices to show that $\cH^p(F^\bdot) \in \D^b_\S(\X)_{<w}$. By Remark \ref{rem_exact_truncation}, we have that $\radj{w} \cH^p(F^\bdot)[-p] \in \D^b_\S(\X)^{\leq p}$. It follows that
$$\op{Hom}(\radj{w} \cH^p(F^\bdot)[-p],\cH^p(F^\bdot)[-p]) \to \op{Hom}(\radj{w} \cH^p(F^\bdot)[-p], F^\bdot)$$
is an isomorphism.

The object $\radj{w} \cH^p(F^\bdot)[-p]$ lies in the category generated by $j_\ast \D^b(\S)_{\geq w}$, so the adjunction $j_\ast \dashv j^!$ and the hypothesis that $j^! F^\bdot \in \D^+(\S)_{<w}$ implies that $\op{Hom}(\radj{w} \cH^p(F^\bdot)[-p], F^\bdot) = 0$. This in turn implies that the canonical map $\radj{w} \cH^p(F^\bdot)[-p] \to \cH^p(F^\bdot)[-p]$ is the zero map. This cannot happen unless $\radj{w} \cH^p(F^\bdot) = 0$, or in other words $\cH^p(F^\bdot) \in \D^b_\S(\X)_{<w}$.
\end{proof}

\begin{proof}[Proof of Theorem \ref{thm_quantization}]
This is equivalent to the vanishing of $R\op{\Gamma}_\S (R \inner{\op{Hom}}_\X (F^\bdot,G^\bdot))$. By the formula
$$Rj^! \inner{\op{Hom}}_\X (F^\bdot, G^\bdot) \simeq \inner{\op{Hom}}_\S(Lj^\ast F^\bdot, Rj^! G^\bdot)$$
it suffices to prove the case where $F^\bdot = \sh{O}_X$, i.e. showing that $R\op{\Gamma}_\S (G^\bdot) = 0$ whenever $G^\bdot \in \D^+(\X)_{<0}$.

Lemma \ref{lem_koszul_system} provides a system $K_1 \to K_2 \to \cdots $ of perfect complexes in $\D^b_\S(\X)_{<1}$ such that $$R\op{\Gamma}_\S (G^\bdot) = \colim  R\Gamma ( K_i^\bdot \otimes G^\bdot ) \simeq \varinjlim \limits_{i,m} R\Gamma( \tau^{\leq m}(K_i^\bdot \otimes G^\bdot)),$$
so it suffices to prove the vanishing for each term in the colimit. We have $j^! ( K^\bdot_i \otimes G^\bdot ) = j^\ast (K^\bdot_i) \otimes j^! G^\bdot$, so $K_i^\bdot \otimes G^\bdot \in \D^+_\S(\X)_{<0}$. Lemma \ref{lem_homology_weights} implies that $\tau^{\leq m} (K_i \otimes G^\bdot) \in \D^b_\S(\X)_{<0}$ for all $m$. Finally, the  category $\D^b_\S(\X)_{<0}$ is generated by objects of the form $j_\ast F^\bdot$ with $F^\bdot \in \D^b(\S)_{<0}$, and thus $R\Gamma (F^\bdot) = 0$ for all $F^\bdot \in \D^b_\S(\X)_{<0}$.

%
\end{proof}

\subsection{Alternative characterizations of $\D^b(\X)_{<w}$}
\label{sect_describe_categories}

Both the categories $\D^b(\X)_{<w}$ and $\D^b_\S(\X)_{<w}$ involve the condition $j^! F^\bdot \in \D^+(\S)_{<w}$. In practice, it is convenient to have alternative ways of describing this condition which explicitly only depend on $F^\bdot$ in a neighborhood of $Z$.

When $\S$ satisfies Property \hyperref[property_A]{(A)}, $\sigma : \Z \to \S$ has finite tor dimension, so $\sigma^\ast j^!$ maps $\D^b(\X)$ to the full subcategory $\D^+(\Z) \subset \D_{qc}(\Z)$. If we use $j$ to also denote the closed immersion $j : S / L \to X/L$, then the diagram
\begin{equation} \label{eqn:forgetful_diagram}
\xymatrix{D^+(X/G) \ar[r]^{j^!} \ar[d] & D^+(S/G) \ar[d] \ar[dr]^{\sigma^\ast} &  \\ D^+(X/L) \ar[r]^{j^!} & D^+(S/L) \ar[r]^{\sigma^\ast} & D_{qc} (Z/L) }
\end{equation}
canonically commutes. Thus $\sigma^\ast j^!$ canonically factors through the pullback functor $D^b(\X) \to D^b(X/L)$, which is the forgetful functor regarding a $G$-equivariant complex as an $L$-equivariant complex.

We will introduce a slight abuse of notation and use $\sigma^!$ to denote the composition
\begin{equation}
\sigma^! : \D^b(\X) \to \D^b(X/L) \xrightarrow{\inner{\op{Hom}}(\cO_\Z,\bullet)} \D^+(\Z).
\end{equation}
Note that this is not the right adjoint of $\sigma_\ast : D^+(\Z) \to D^+(\X)$, as $\sigma : \Z \to \X$ is not a closed immersion. However, it has many of the same formal properties, such as the formula $\sigma^! (K^\bdot \otimes F^\bdot) \simeq \sigma^\ast(K^\bdot) \otimes \sigma^! (F^\bdot)$ for $K^\bdot \in \op{Perf}(\X)$.

\begin{prop} \label{prop:describe_categories}
Let $F^\bdot \in \D^b(\X)$ and assume that $\S$ satisfies Property \hyperref[property_A]{(A)}. Then the following are equivalent
\begin{enumerate}
\item $\sigma^! F^\bdot \in \D^+(\Z)_{<w+a}$
\item $\sigma^\ast j^! F^\bdot \in \D^+(\Z)_{<w}$
\item $j^! F^\bdot \in \D^+(\S)_{<w}$
\end{enumerate}
Where we define the integer
\begin{equation} \label{eqn:define_a}
a := \op{weight}_\lambda \det(\fN_Z Y) + \op{weight}_\lambda \det(\lie{g}_{\lambda<0}).
\end{equation}
\end{prop}

\noindent Before we prove this proposition, we prove the following
\begin{lem}
If $\S$ satisfies Property \hyperref[property_A]{(A)}, then for $F^\bdot \in \D^b(\X)$,
\begin{equation} \label{eqn_compare_shriek_ast}
\sigma^! (F^\bdot) \simeq \sigma^\ast j^! (F^\bdot) \otimes \det(\fN_Z Y) \otimes \det(\lie{g}_{\lambda<0})[-c],
\end{equation}
where $c$ is the codimension of $Z \hookrightarrow S$.
\end{lem}
\begin{proof}
From the commutative diagram \eqref{eqn:forgetful_diagram}, it suffices to pass to $X/L$ via the forgetful functor, and thus we may regard $\sigma^!$ as the usual shriek-pullback for the closed immersion $Z/L \to X/L$.

Property \hyperref[property_A]{(A)} implies that $Z \hookrightarrow Y$ is a regular embedding. Furthermore Property \ref{property_S_2} guarantees that $Y \simeq P \times_P Y \hookrightarrow S \simeq G \times_P Y$ is a regular embedding with normal bundle $\cO_Y \otimes \lie{g}_{\lambda<0}$, where the $P$-equivariant structure is given by the adjoint action of $P$ on $\lie{g}_{\lambda<0} = \lie{g} / \lie{g}_{\geq 0}$. It follows that \eqref{eqn_compare_shriek_ast} holds, with $\sigma^\ast$ instead of $\sigma^\ast j^!$, as an isomorphism of functors $\D^+(S/L) \to \D^+(Z/L)$. The claim follows by pre-composing with $j^! : \D^b(X/L) \to \D^+(S/L)$. 
\end{proof}

We also observe a dual form of Nakayama's lemma:
\begin{lem} \label{lem:dual_Nakayama}
If $F^\bdot \in \D^+(\S)$ and $\sigma^! F^\bdot \simeq 0$, then $F^\bdot \simeq 0$.
\end{lem}
\begin{proof}
For the moment consider $F^\bdot$ to be an $L$-equivariant complex on $S$ via the forgetful functor. Let $\omega_S$ be an  $L$-equivariant dualizing complex on $S$, and $\omega_Z = \sigma^! \omega_S$ the corresponding dualizing complex on $Z$. We have that $F^\bdot \in \D^+(S/L)$ is zero if and only if $\gd (F^\bdot) \in \D^-(S/L)$ is zero. Furthermore $\sigma^! F^\bdot \simeq \gd (\sigma^\ast \gd (F^\bdot))$, so if $\sigma^! F^\bdot = 0$, then $\sigma^\ast \gd(F^\bdot) = 0$, and by Nakayama's Lemma, \ref{lem:Nakayama}, each homology sheaf of $F^\bdot \in \D^b(S/L)$ vanishes in a neighborhood of $Z$. Because each $\cH^i(F^\bdot)$ is actually a $G$-quivariant sheaf, and $S$ is the only $G$-equivariant open subset containing $Z$, $\cH^i(F^\bdot) = 0$ for all $i$.
\end{proof}

\begin{proof}[Proof of Proposition \ref{prop:describe_categories}]
The equivalence between (1) and (2) follows immediately from the formula \eqref{eqn_compare_shriek_ast} and the fact that the expression \eqref{eqn:define_a} is the $\lambda$-weight of the invertible sheaf $\det(\fN_Z Y) \otimes \det(\lie{g}_{\lambda<0})$. To show that (2) is equivalent to (3) it suffices to show that $G^\bdot \in \D^+(\S)$ lies in $\D^+(\S)_{<w}$ if and only if $\sigma^\ast G^\bdot \in \D_{qc}(\Z)_{<w}$.

\medskip
\noindent $(3) \Rightarrow (2):$
\medskip

If $G^\bdot \in \D^+(\S)_{<w}$, then
$$\cH^p(\sigma^\ast G^\bdot) \simeq \cH^p(\sigma^\ast \tau^{\geq p} G^\bdot) \simeq \varinjlim_m \cH^p(\sigma^\ast \tau^{\leq m} \tau^{\geq p} G^\bdot).$$
This lies in $\qcoh{\Z}_{<w}$ because each $\tau^{\leq m} \tau^{\geq p} G^\bdot \in \D^b(\S)_{<w}$ and $\D_{qc}(\Z)_{<w}$ is closed under colimits. As remarked above, Property \hyperref[property_A]{(A)} implies that $\sigma^\ast G^\bdot \in \D^+(\Z)$ as well.

\medskip
\noindent $(2) \Rightarrow (3):$
\medskip

Let $G^\bdot \in \D^+(\S)$ and assume that $\sigma^\ast G^\bdot \in \D_{qc}(\Z)_{<w}$. Property \hyperref[property_A]{(A)} implies that $\sigma^\ast$ has finite tor dimension, so $\sigma^\ast \tau^{\leq m} G^\bdot$ agrees with $\sigma^\ast G^\bdot$ in low cohomological degree relative to $m$. It follows that for any $l > 0$, we can choose $m \gg l$ such that
$$\sigma^\ast \radj{w} \tau^{\leq m} G^\bdot \simeq \radj{w} \sigma^\ast \tau^{\leq m} G^\bdot \in \D^b(\S)^{\geq l}$$
Consequently $\sigma^! \radj{w} \tau^{\leq m} G^\bdot \in \D^b(\Z)^{\geq l-c}$, where $c$ is the codimension of $Z \hookrightarrow Y$. Lemma \ref{lem:dual_Nakayama} implies that $\radj{w} \tau^{\leq m} G^\bdot \in \D^b(\S)^{\geq l-c}$.

Because we could have chosen $l$ and $m$ arbitrarily large, we have that for any $E^\bdot \in \D^b(\S)_{\geq w}$,
$$\op{Hom}(E^\bdot,G^\bdot) \simeq \op{Hom}(E^\bdot,\tau^{\leq m} G^\bdot) \simeq \op{Hom} (E^\bdot, \radj{w} \tau^{\leq m} G^\bdot) \simeq 0.$$
Finally, if $G^\bdot \in \D^+(\S)$ and $\op{Hom}(E^\bdot,G^\bdot) = 0$ for all $E^\bdot \in \D^b(\S)_{\geq w}$, then $G^\bdot \in \D^+(\S)_{<w}$. To show this one proceeds inductively by showing that the lowest homology sheaf of $G^\bdot$ must lie in $\coh{\S}_{<w}$, or else it would receive a non-zero map from $\cA \otimes (\cH^{min}(G^\bdot))_h$.

\end{proof}

\begin{rem}
The functors $\sigma^\ast,\sigma^! : \D^b(X/L) \to \D^\pm(Z/L)$ commute with the forgetful functors $\D^b(X/L) \to \D^b(X/\lambda(\Gm))$ and $\D^\pm (Z/L) \to \D^\pm(Z/\lambda(\Gm))$. Therefore assuming \hyperref[property_A]{(A)}, an object $F^\bdot \in \D^b(\X)$ lies in $\D^b(\X)_{<w}$ (respectively $\D^b(\X)_{\geq w}$) if and only if when we forget all but the equivariance with respect to $\Gm$, we have $\sigma^\ast F^\bdot \in \D^-(Z/\lambda(\Gm))_{\geq w}$ (respectively $\sigma^!F^\bdot \in \D^+(Z/\lambda(\Gm))_{<w+a}$).

One can even refine this to a point-wise criterion. Each point of $Z$ defines a morphism of stacks $p : \pt / \Gm \to \X$. Using Nakayama's Lemma, \eqref{lem:Nakayama}, one can show that $F^\bdot \in \D^b(\X)_{\geq w}$ if and only if $p^\ast F^\bdot \in \D^-(\pt / \Gm)_{\geq w}$ for all points in $Z$. Likewise we can define
$$p^!(F^\bdot) := \op{Hom}^\bdot_X (p_\ast k, F^\bdot |_{X/\lambda(\Gm)}) \in \D^+(\pt / \Gm)$$
for each point of $Z$, and using the dual form of Nakayama's lemma, one can show that $F^\bdot \in \D^b(\X)$ lies in $\D^b(\X)_{<w}$ if and only if $p^!(F^\bdot) \in \D^+(\pt / \Gm)_{<w+a}$ for all points in $Z$. We omit the proofs of these facts, as we will not explicitly use them here.
\end{rem}

\subsection{Semiorthogonal decomposition of $\D^b(\X)$}
\label{sect_main_semi_decomp}

In this section we construct the semiorthogonal decomposition of $\D^b(\X)$ used to prove the categorical Kirwan surjectivity theorem. We will prove:

\begin{thm} \label{thm_main_semi_decomp}
Let $\S \subset \X$ be a closed KN stratum (Definition \ref{def_KN_stratification}) satisfying Properties \hyperref[property_L_plus]{(L+)} and \hyperref[property_A]{(A)}. Let $\G_w = \D^b(\X)_{\geq w} \cap \D^b(\X)_{<w}$, then
$$\G_w = \left\{ F^\bdot \in \D^b(\X) \left| \begin{array}{c}  \lambda \text{-weights of } \cH^\ast(\sigma^\ast F^\bdot) \text{ are } \geq w,\text{ and} \\ \lambda \text{-weights of } \cH^\ast(\sigma^! F^\bdot) \text{ are } <w+a. \end{array} \right. \right\}$$
where $a$ is defined as in \eqref{eqn:define_a}. There are semiorthogonal decompositions
$$\D^b(\X) = \langle \D^b_\S(\X)_{<w}, \G_w , \D^b_\S(\X)_{\geq w} \rangle$$
And the restriction functor $i^\ast : \D^b(\X) \to \D^b(\V)$ induces an equivalence $\G_w \simeq \D^b(\V)$, where $\V = \X - \S$.
\end{thm}

The first key observation is that if $\S \subset \X$ is a KN stratum satisfying Property \hyperref[property_A]{(A)}, then for any $F^\bdot \in \D^b(\X)$, the weights of $\sigma^\ast F^\bdot$ are bounded below, and the weights of $\sigma^! F^\bdot$ are bounded above.

\begin{lem} \label{lem_bounded_weights}
Let $\S \subset \X$ be a KN stratum satisfying Property \hyperref[property_A]{(A)}. Then
$$\D^b(\X) = \bigcup_{v<w} \left( \D^b(\X)_{\geq v} \cap \D^b(\X)_{<w} \right).$$
\end{lem}

\begin{proof}
Let $K_i^\bdot$ be a Koszul system for $\S$. Lemma \ref{lem:alternate_Koszul_property} implies that for $i \gg 0$ we have that $\op{Cone}(\sigma^\ast K^\bdot_i \to \cO_\Z)| \in \D^b(\Z)_{<0}$. Because objects of $\D^b(\Z)$ decompose canonically into a direct sum of weight spaces under the $\lambda$-action, it follows that $\cO_\Z$ is actually a direct summand of $\sigma^\ast K^\bdot_i$.

Now let $F^\bdot \in \D^b(\X)$. Because $K^\bdot_i \otimes F^\bdot$ is supported on $\S$, it lies in $\D^b_\S(\X)_{\geq a} \cap \D^b_\S(\X)_{<b}$ for some $a$ and $b$ because the baric decomposition of $\D^b_\S(\X)$ is bounded. This is equivalent to the weights of $\sigma^\ast (K^\bdot_i \otimes F^\bdot) \simeq \sigma^\ast K^\bdot_i \otimes \sigma^\ast F^\bdot$ being bounded below and, by Proposition \ref{prop:describe_categories}, the weights of $\sigma^!(K_i^\bdot \otimes F^\bdot) \simeq \sigma^\ast K^\bdot_i \otimes \sigma^! F^\bdot$ being bounded above. Because $\cO_\Z$ is a direct summand of $\sigma^\ast K^\bdot_i$, it follows that the weights of $\sigma^\ast F^\bdot$ are bounded below and the weights of $\sigma^! F^\bdot$ are bounded above.

\end{proof}

%

Using this result, we explicitly construct right adjoints for each of the inclusions $\D^b_\S(\X)_{\geq w} \subset \D^b(\X)_{\geq w} \subset \D^b(\X)$.

\begin{lem} \label{lem_radj}
Let $\S \subset \X$ be a KN stratum satisfying Property \hyperref[property_A]{(A)}, and let $K_i^\bdot$ be a Koszul system for $\S$. Let $F^\bdot \in \D^b(\X)$. Then for sufficiently large $i$ the canonical map
$$\radj{w} (K_i^\bdot \otimes F^\bdot) \to  \radj{w} (K^\bdot_{i+1} \otimes F^\bdot)$$
is an equivalence. It follows that the complex
\begin{equation} \label{eqn_define_radj}
\radj{w} \underline{\Gamma}_\S (F^\bdot) := \varinjlim_i \radj{w} \left( K^\bdot_i \otimes F^\bdot \right)
\end{equation}
lies in $\D^b_\S(\X)_{\geq w}$. The functor $\radj{w} \inner{\Gamma}_\S$, defined by \eqref{eqn_define_radj}, is a right adjoint to the inclusions $\D^b_\S(\X)_{\geq w} \subset \D^b(\X)_{\geq w}$ and $\D^b_\S(\X)_{\geq w} \subset \D^b(\X)$.
\end{lem}

\begin{proof}
By hypothesis the $C^\bdot_i := \op{Cone}(K^\bdot_i \to K^\bdot_{i+1})$ is a perfect complex in $\D^b_\S(\X)_{<w_i}$, where $w_i \to -\infty$ as $i \to \infty$. By Lemma \ref{lem_bounded_weights}, we have $F^\bdot \in \D^b(\X)_{<N}$ for some $N$, so if $w_i + N < w$ we have $C_i^\bdot \otimes F^\bdot \in \D^b_\S(\X)_{<w}$ and
$$\op{Cone}\left( \radj{w} (K_i^\bdot \otimes F^\bdot) \to \radj{w} (K_{i+1}^\bdot \otimes F^\bdot) \right) = \radj{w} (C_i^\bdot \otimes F^\bdot) = 0.$$
Thus the direct system $\radj{w} (K_i^\bdot \otimes F^\bdot)$ stabilizes, and the expression \eqref{eqn_define_radj} defines a functor $\D^b(\X) \to \D^b_\S(\X)_{\geq w}$.

The fact that $\radj{w} \inner{\Gamma}_\S$ is the right adjoint of the inclusion follows from the fact that elements of $\D^b_\S(\X)$ are compact in $\D^+_{\S,qc}(\X)$ \cite{Ro08}. For $G^\bdot \in \D^b_\S(\X)_{\geq w}$, we compute
$$R\op{Hom}(G^\bdot,\radj{w} \inner{\op{\Gamma}}_\S F^\bdot) \simeq \varinjlim_i R \op{Hom} (G^\bdot, K_i^\bdot \otimes F^\bdot) \simeq R \op{Hom}(G^\bdot, F^\bdot).$$
\end{proof}

Because $\D^b_\S(\X)_{\geq w}$ is generated by $j_\ast \D^b(\S)_{\geq w}$, the right orthogonal to $\D^b_\S(\X)_{\geq w}$ is $\D^b(\X)_{<w}$. Lemma \ref{lem_radj} gives semiorthogonal decompositions
$$\D^b(\X)_{\geq w} = \langle \G_w, \D^b_\S(\X)_{\geq w} \rangle \qquad \D^b(\X) = \langle \D^b(\X)_{<w}, \D^b_\S(\X)_{\geq w} \rangle$$
where $\G_w := \D^b(\X)_{\geq w} \cap \D^b(\X)_{<w}$. What remains is to show that $\D^b(\X)_{\geq w} \subset \D^b(\X)$ is right admissible.


\begin{lem} \label{lem_radj2}
The inclusion of the subcategory $\D^b(\X)_{\geq w} \subset \D^b(\X)$ admits a right adjoint $\radj{w}(\bullet)$ defined by the exact triangle
$$\radj{w} F^\bdot \to F^\bdot \to \ladj{w} \left( (K^\bdot_i)^\dual \otimes F^\bdot \right) \parr \qquad \text{for } i \gg 0$$
\end{lem}

\begin{proof}
First note that for $F^\bdot \in \D^b(\X)$ and for $i\gg 0$, $\op{Cone}(K_i^\bdot \to K_{i+1}^\bdot)^\dual \otimes F^\bdot \in \D^b(\X)_{\geq w}$. It follows that the inverse system $(K^\bdot_i)^\dual \otimes F^\bdot$ stabilizes, as in Lemma \ref{lem_radj}.

Consider the composition $F^\bdot \to (K_i^\bdot)^\dual \otimes F^\bdot \to \ladj{w}((K_i^\bdot)^\dual \otimes F^\bdot)$. With $\radj{w} F^\bdot$ defined as above, the octahedral axiom gives an exact triangle
$$\op{Cone}(F^\bdot \to (K_i^\bdot)^\dual \otimes F^\bdot) \to \radj{w} F^\bdot [1] \to \radj{w}( (K_i^\bdot)^\dual \otimes F^\bdot ) \parr .$$
Lemmas \ref{lem:alternate_Koszul_property} and \ref{lem_bounded_weights} imply that for $i\gg 0$, $\radj{w} F^\bdot \in \D^b(\X)_{\geq w}$ and is thus right orthogonal to $\D^b_\S(\X)_{<w}$. It follows that $\radj{w} F^\bdot$ is functorial in $F^\bdot$ and this functor is a right adjoint to the inclusion $\D^b(\X)_{\geq w} \subset \D^b(\X)$.
\end{proof}

\begin{proof}[Proof of Theorem \ref{thm_main_semi_decomp}]
The existence of the semiorthogonal decomposition follows formally from the adjoint functors constructed in Lemmas \ref{lem_radj} and \ref{lem_radj2}. The fully-faithfulness of $i^\ast : \G_w \to \D^b(\V)$ is Theorem \ref{thm_quantization}. Any $F^\bdot \in \D^b(\V)$ admits a lift to $\D^b(\X)$, and the component of this lift lying in $\G_w$ under the semiorthogonal decomposition also restricts to $F^\bdot$, hence $i^\ast : \G_w \to \D^b(\V)$ is essentially surjective.
\end{proof}

Recall from Lemma \ref{lem:smooth_strata_2} that when $\X$ is smooth in a neighborhood of $\Z$, the condition that $\sigma^\ast j^! F^\bdot \in \D_{qc}(\Z)_{<w}$ is equivalent to the condition that $\sigma^\ast F^\bdot \in \D_{qc}(\Z)_{<w+\eta}$, where we define $\eta := \op{weight}_\lambda \det(\coNorm_{\S} \X)|_\Z$. This allows us to restate our main theorem in a form which will be convenient for the applications in Section \ref{sect_var_GIT}.

\begin{cor} \label{thm_main_semi_decomp_smooth}
Let $\S \subset \X$ be a KN stratum such that $\X$ is smooth in a neighborhood of $\Z$. Let $\G_w = \D^b(\X)_{\geq w} \cap \D^b(\X)_{<w}$, then
$$\G_w = \{ F^\bdot \in \D^b(\X) | \lambda \text{-weights of } \cH^\ast ( \sigma^\ast F^\bdot) \text{ lie in } [w,w+\eta). \}$$
where $\eta$ is the weight of $\det (\coNorm_{\S} \X)$. There are semiorthogonal decompositions
$$\D^b(\X) = \langle \D^b_\S(\X)_{<w}, \G_w , \D^b_\S(\X)_{\geq w} \rangle$$
And the restriction functor $i^\ast : \D^b(\X) \to \D^b(\V)$ induces an equivalence $\G_w \simeq \D^b(\V)$.
\end{cor}

\section{Derived equivalences and variation of GIT}
\label{sect_var_GIT}

We apply Theorem \ref{thm_kirwan_surj} to the derived categories of birational varieties obtained by a variation of GIT quotient. First we study the case where $G=\Gm$, in which the KN stratification is particularly easy to describe. Next we generalize this analysis to arbitrary variations of GIT, one consequence of which is the observation that if a smooth projective-over-affine variety $X$ is equivariantly Calabi-Yau for the action of a torus, then the GIT quotient of $X$ with respect to any two generic linearizations are derived-equivalent.

A normal projective variety $X$ with linearized $\Gm$ action is sometimes referred to as a birational cobordism between $X//_{\sh{L}} G$ and $X//_{\sh{L}(r)} G$ where $\sh{L}(m)$ denotes the twist of $\sh{L}$ by the character $t \mapsto t^r$. A priori this seems like a highly restrictive type of VGIT, but by Thaddeus' master space construction \cite{Th96}, any two spaces that are related by a general VGIT are related by a birational cobordism. We also have the weak converse due to Hu \& Keel:
\begin{thm}[Hu \& Keel]
Let $Y_1$ and $Y_2$ be two birational projective varieties, then there is a birational cobordism $X/\Gm$ between $Y_1$ and $Y_2$. If $Y_1$ and $Y_2$ are smooth, then by equivariant resolution of singularities $X$ can be chosen to be smooth.
\end{thm}

The GIT stratification for $G=\Gm$ is very simple. If $\sh{L}$ is chosen so that the GIT quotient is an orbifold, then the $Z_i$ are the connected components of the fixed locus $X^G$, and $S_i$ is either the ascending or descending manifold of $Z_i$, depending on the weight of $\sh{L}$ along $Z_i$.

We will denote the tautological choice of \oneps\ as $\lambda^+$, and we refer to ``the weights'' of a coherent sheaf at point in $X^G$ as the weights with respect to this \oneps. We define $\mu_i \in \bZ$ to be the weight of $\sh{L}|_{Z_i}$. If $\mu_i <0$ (respectively $\mu_i >0$) then the maximal destabilizing \oneps\ of $Z_i$ is $\lambda^+$ (respectively $\lambda^-$). Thus we have
$$S_i = \left\{ x \in X \left| \begin{array}{c} \lim \limits_{t \to 0} t \cdot x \in Z_i \text{ if } \mu_i < 0 \\ \lim \limits_{t\to 0} t^{-1} \cdot x \in Z_i \text{ if } \mu_i > 0 \end{array} \right. \right\}$$
Next observe the weight decomposition under $\lambda^+$:
\begin{equation} \label{eqn_cotangent_weight_decomp}
\Omega_X^1|_{Z_i} \simeq \Omega^1_{Z_i} \oplus \sh{N}^+ \oplus \sh{N}^-
\end{equation}
Then $\Omega^1_{S_i}|_{\Z_i} = \Omega^1_{Z_i} \oplus \sh{N}^-$ if $\mu_i < 0$ and $\Omega^1_{S_i}|_{\Z_i} = \Omega^1_{Z_i} \oplus \sh{N}^+$ if $\mu_i>0$, so we have
\begin{equation} \label{eqn_window_width}
\eta_i = \left\{ \begin{array}{c} \text{weight of } \det \sh{N}^+|_{Z_i} \text{ if } \mu_i < 0 \\ -\text{weight of } \det \sh{N}^-|_{Z_i} \text{ if } \mu_i > 0 \end{array} \right.
\end{equation}

\begin{figure}
\centering
\includegraphics[scale=0.7]{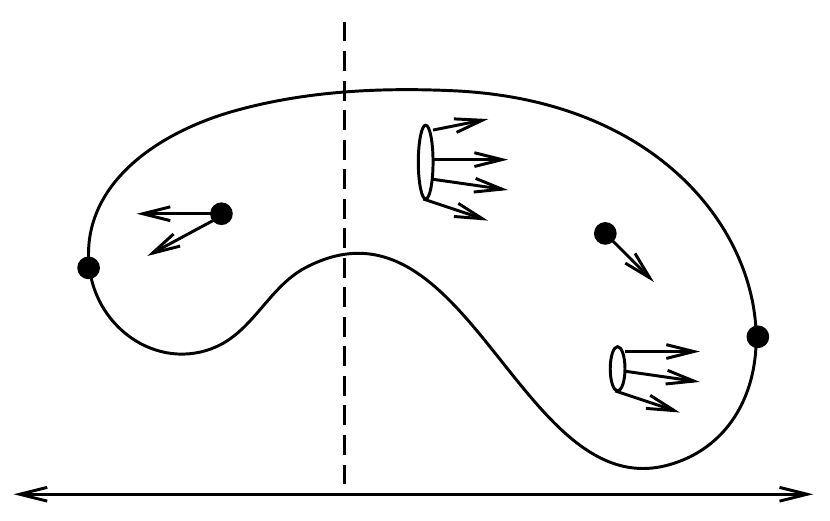}
\quad
\includegraphics[scale=0.7]{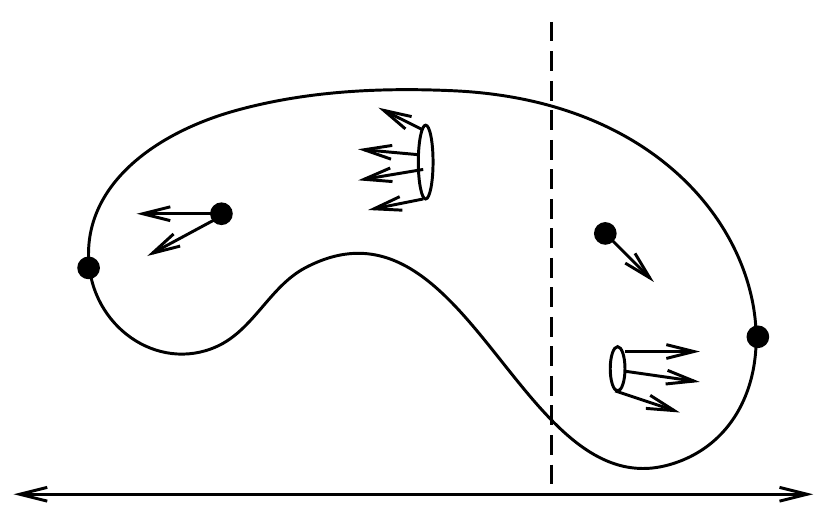}
\caption{Schematic diagram for the fixed loci $Z_i$. $S_i$ is the ascending or descending manifold of $Z_i$ depending on the sign of $\mu_i$. As the moment fiber varies, the unstable strata $S_i$ flip over the critical sets $Z_i$.}
\end{figure}

There is a parallel interpretation of this in the symplectic category when the base field $k = \bC$. A sufficiently large power of $\sh{L}$ induces a equivariant projective embedding and thus a moment map $\mu : X \to \bR$ for the action of $S^1 \subset \bC^\ast$. The semistable locus is the orbit of the zero fiber $X^{ss} = \bC^\ast \cdot \mu^{-1} (0)$. The reason for the collision of notation is that the fixed loci $Z_i$ are precisely the critical manifolds of $\mu$, and the number $\mu_i$ is the value of the moment map on the critical set $Z_i$.

Varying the linearization $\sh{L}(r)$ by twisting by the character $t \mapsto t^{-r}$ corresponds to shifting the moment map by $r$, so the new zero fiber corresponds to what was previously the fiber $\mu^{-1}(-r)$. For non-critical moment fibers the GIT quotient will be a DM stack, and the critical values of $r$ are those for which the weight of $\sh{L}(r)|_{Z_i} = 0$ for some $i$.

Now we return to a general base field $k$. Say that as $r$ increases it crosses a critical value for which $\cL(r)|_{Z_i}$ has weight $0$, so $\mu_i > 0 > \mu_i^\prime$. The maximal destabilizing \oneps, $\lambda_i$, flips from $\lambda^-$ to $\lambda^+$, and the unstable stratum, $S_i$, flips from the descending manifold of $Z_i$ to the ascending manifold of $Z_i$. We apply \eqref{eqn_window_width} to compute
\begin{equation} \label{eqn_wall_crossing}
\eta_i^\prime - \eta_i = \text{weight of } \omega_X|_{Z_i}
\end{equation}

Thus if $\omega_X$ has weight $0$ along $Z_i$, the integer $\eta_i$ does not change as we cross the wall. The grade restriction window of Theorem \ref{thm_kirwan_surj} has the same width for the GIT quotient on either side of the wall, and it follows that the two GIT quotients are derived equivalent because they are identified with the \emph{same} subcategory $\G_w$ of the equivariant derived category $\D^b(X/G)$. We summarize this with the following
\begin{prop} \label{prop_birat_cobord}
Let $\sh{L}$ be a critical linearization of $X/\Gm$, and assume that $Z_i$ is the only critical set for which $\mu_i = 0$. Let $a$ be the weight of $\omega_X|_{Z_i}$, and let $\epsilon > 0$ be a small rational number.
\begin{enumerate}
\item If $a<0$, then there is a fully faithful embedding
$$\D^b(X //_{\sh{L}(\epsilon)} G) \subseteq \D^b(X //_{\sh{L}(-\epsilon)} G)$$
\item If $a=0$, then there is an equivalence
$$\D^b(X //_{\sh{L}(\epsilon)} G) \simeq \D^b(X //_{\sh{L}(-\epsilon)} G)$$
\item If $a>0$, then there is a fully faithful embedding
$$\D^b(X //_{\sh{L}(-\epsilon)} G) \subseteq \D^b(X //_{\sh{L}(\epsilon)} G)$$
\end{enumerate}
\end{prop}

The analytic local model for a birational cobordism is the following
\begin{ex}\label{ex_local_model}
Let $Z$ be a smooth variety and let $\cN = \bigoplus \cN_i$ be a $\bZ$-graded locally free sheaf on $Z$ with $\cN_0 = 0$. Let $X$ be the total space of $\cN$ -- it has a $\Gm$ action induced by the grading. Because the only fixed locus is $Z$ the underlying line bundle of the linearization is irrelevant, so we take the linearization $\cO_X (r)$.
\end{ex}

If $r < 0$ then the unstable locus is $\cN_- \subset X$ where $\cN_-$ is the sum of negative weight spaces of $\cN$, and if $r>0$ then the unstable locus is $\cN_+$ (we are abusing notation slightly by using the same notation for the sheaf and its total space). We will borrow the notation of Thaddeus \cite{Th96} and write $X/\pm = (X \setminus \cN_\pm) / \Gm$.

Inside $X/\pm$ we have $\cN_\mp / \pm \simeq \bP(\cN_\mp)$, where we are still working with quotient stacks, so the notation $\bP(\cN_\mp)$ denotes the weighted projective bundle associated to the graded locally free sheaf $\cN_\mp$. If $\pi_\mp : \bP(\cN_\mp) \to Z$ is the projection, then $X/\pm$ is the total space of the vector bundle $\pi_\mp^\ast \cN_\mp (-1)$. We have the common resolution
$$\xymatrix@C=10pt{ & \cO_{\bP(\cN_-) \times_S \bP(\cN_+)}(-1,-1) \ar[dl] \ar[dr] & \\ \pi_+^\ast \cN_-(-1) & & \pi_-^\ast \cN_+(-1) }$$

Let $\pi : X \to Z$ be the projection, then the canonical bundle is $\omega_X = \pi^\ast ( \omega_Z \otimes \det (\cN_+)^\dual \otimes \det (\cN_-)^\dual )$, so the weight of $\omega_X|_Z$ is $\sum i \op{rank} (\cN_i)$. In the special case of a flop, Proposition \ref{prop_birat_cobord} says:
$$\text{if } \sum i \op{rank} (\cN_i) = 0,  \text{ then } \D^b(\pi_+^\ast \cN_-(-1)) \simeq \D^b(\pi_-^\ast \cN_+(-1)).$$

\subsection{General variation of GIT quotient}

We will generalize the analysis of a birational cobordism to a special kind of variation of GIT quotient which we will call \emph{balanced}. Until this point we have taken the KN stratification as given, but now we must recall its definition and basic properties as described in \cite{DH98}.

Let $\op{NS}^G(X)_\bR := \op{NS}^G(X) \otimes \bR$ denote the extension of scalars of the group of equivariant invertible sheaves modulo homological equivalence. For any $\cL \in \op{NS}^G(X)_\bR$ one defines a stability function on $X$
$$M^\cL (x) := \max \left\{ \left. \frac{-\op{weight}_\lambda \cL_y }{|\lambda|} \right| \lambda \text{ s.t. } y = \lim_{t \to 0} \lambda(t)\cdot x 
\text{ exists} \right\}$$
$M^\cL(\bullet)$ is upper semi-continuous, and $M^\bullet(x)$ is lower convex and thus continuous on $\op{NS}^G(X)_\bR$ for a fixed $x$. A point $x \in X$ is semistable if $M^\cL(x) \leq 0$, stable if $M^\cL(x) < 0$ and $\op{Stab}(x)$ is finite, strictly semistable if $M^\cL(x) = 0$, and unstable if $M^\cL(x) > 0$.

The $G$-ample cone $\cC^G(X) \subset \op{NS}^G(X)_\bR$ has a finite decomposition into convex conical chambers separated by hyperplanes -- the interior of a chamber is where $M^\cL(x) \neq 0$ for all $x \in X$, so $\X^{ss} (\cL) = \X^{s} (\cL)$. We will focus on a single wall-crossing: $\cL_0$ will be a $G$-ample invertible sheaf lying on a wall such that for $\epsilon$ sufficiently small $\cL_\pm := \cL_0 \pm \epsilon \cL^\prime$ both lie in the interior of chambers.

Because the function $M^\bullet(x)$ is continuous on $\op{NS}^G(X)_\bR$, all of the stable and unstable points of $\X^s(\cL_0)$ will remain so for $\cL_\pm$. Only points in the strictly semistable locus, $\X^{sss}(\cL_0) = \{ x \in \X | M^\cL(x) = 0 \} \subset \X$, change from being stable to unstable as one crosses the wall.

In fact $\X^{us}(\cL_0)$ is a union of KN strata for $\X^{us}(\cL_+)$, and symmetrically it can be written as a union of KN strata for $\X^{us}(\cL_-)$ \cite{DH98}. Thus we can write $\X^{ss}(\cL_0)$ in two ways
\begin{equation} \label{eqn_sss_decomp}
\X^{ss}(\cL_0) = \S^\pm_1 \cup \cdots \cup \S^\pm_{m_\pm} \cup \X^{ss}(\cL_\pm)
\end{equation}
Where $\S^\pm_i$ are the KN strata of $\X^{us}(\cL_\pm)$ lying in $\X^{ss}(\cL_0)$.

\begin{defn} \label{def_balanced_wall_crossing}
A wall crossing determined by $\cL_\pm = \cL_0 \pm \epsilon \cL^\prime$ will be called \emph{balanced} if $m_+ = m_-$ and $\Z^+_i \simeq \Z^-_i$ under the decomposition \eqref{eqn_sss_decomp}.
\end{defn}

By the construction of the KN stratification in GIT, we can further rigidify this picture. One can find $\lambda_i$ and locally closed $Z_i \subset X^{\lambda_i}$ such that the $\lambda_i^{\pm}$ are distinguished \oneps's for the KN strata $\S_i^\pm$, and $\Z_i^+ = Z_i / L_i = \Z_i^-$.

\begin{prop} \label{prop_var_git}
Let a reductive $G$ act on a projective-over-affine variety $X$. Let $\cL_0$ be a $G$-ample line bundle on a wall, and define $\cL_\pm = \cL_0 \pm \epsilon \cL^\prime$ for some other line bundle $\cL^\prime$. Assume that
\begin{itemize}
\item for $\epsilon$ sufficiently small, $\X^{ss}(\cL_\pm) = \X^s (\cL_\pm) \neq \emptyset$,
\item the wall crossing $\cL_\pm$ is balanced, and
\item for all $Z_i$ in $\X^{ss}(\cL_0)$, $(\omega_\X)|_{Z_i}$ has weight $0$ with respect to $\lambda_i$.
\end{itemize}
Then $\D^b(\X^{ss}(\cL_+)) \simeq \D^b(\X^{ss}(\cL_-))$.
\end{prop}

\begin{rem}
Full embeddings analagous to those of Proposition \ref{prop_birat_cobord} apply when the weights of $(\omega_\X)|_{Z_i}$ with respect to $\lambda_i$ are either all negative or all positive.
\end{rem}

\begin{proof}
This is an immediate application of Theorem \ref{thm_kirwan_surj} to the open substack $\X^s(\cL_\pm) \subset \X^{ss}(\cL_0)$ whose complement admits the KN stratification \eqref{eqn_sss_decomp}. Because the wall crossing is balanced, $Z^+_i = Z^-_i$ and $\lambda_i^-(t) = \lambda_i^+(t^{-1})$, and the condition on $\omega_\X$ implies that $\eta^+_i = \eta^-_i$. So Theorem \ref{thm_kirwan_surj} identifies the category $\G_w \subset \D^b(\X^{ss}(\cL_0))$ with both $\D^b(\X^s(\cL_-))$ and $\D^b(\X^s(\cL_+))$.
\end{proof}

\begin{ex}
Dolgachev and Hu study wall crossings which they call \emph{truly faithful}, meaning that the identity component of the stabilizer of a point with closed orbit in $\X^{ss}(\cL_0)$ is $\bC^\ast$. They show that every truly faithful wall crossing is balanced \cite[Lemma 4.2.3]{DH98}.
\end{ex}

Dolgachev and Hu also show that for the action of a torus $T$, there are no codimension $0$ walls and all codimension 1 walls are truly faithful. Thus any two chambers in $\cC^T(X)$ can be connected by a finite sequence of balanced wall crossings, and we have

\begin{cor} \label{cor_all_quotients_equivalent}
Let $X$ be a projective-over-affine variety with an action of a torus $T$. Assume $X$ is equivariantly Calabi-Yau in the sense that $\omega_X \simeq \cO_X$ as an equivariant $\cO_X$-module. If $\cL_0$ and $\cL_1$ are $G$-ample invertible sheaves such that $\X^s(\cL_i) = \X^{ss}(\cL_i)$, then $\D^b(\X^s(\cL_0)) \simeq \D^b(\X^s(\cL_1))$.
\end{cor}

A compact projective manifold with a non-trivial $\bC^\ast$ action is never equivariantly Calabi-Yau, but Corollary \ref{cor_all_quotients_equivalent} applies to a large class of non compact examples. The simplest are linear representations $V$ of $T$ such that $\det V$ is trivial. More generally we have

\begin{ex}
Let $T$ act on a smooth projective Fano variety $X$, and let $\cE$ be an equivariant ample locally free sheaf such that $\det \cE \simeq \omega_X^\dual$. Then the total space of the dual vector bundle $Y = \op{Spec}_X (S^\ast \cE)$ is equivariantly Calabi-Yau and the canonical map $Y \to \op{Spec}(\Gamma(X,S^\ast \cE))$ is projective, so $Y$ is projective over affine and by Corollary \ref{cor_all_quotients_equivalent} any two generic GIT quotients $Y//T$ are derived-equivalent.
\end{ex}

When $G$ is non-abelian, the chamber structure of $\cC^G(X)$ can be more complicated. There can be walls of codimension 0, meaning open regions in the interior of $\cC^G(X)$ where $\X^s \neq \X^{ss}$, and not all walls are truly faithful \cite{DH98}. Still, there are examples where derived Kirwan surjectivity can give derived equivalences under wall crossings which are not balanced.

\begin{defn}
A wall crossing, determined by $\cL_\pm = \cL_0 \pm \epsilon \cL^\prime$, will be called \emph{almost balanced} if under the decomposition \eqref{eqn_sss_decomp}, $m_+ = m_-$ and one can choose maximal destabilizers such that $\lambda_i^- = (\lambda_i^+)^{-1}$ and the closures of $Z_i^+$ and $Z_i^-$ agree.
\end{defn}

In an almost balanced wall crossing for which $\omega_X|_{Z_i}$ has weight $0$ for all $i$, we have the following general principal for establishing a derived equivalence:
\begin{ansatz} \label{ansatz_flop}
One can choose $w$ and $w^\prime$ such that $\G^+_w = \G^-_{w^\prime}$ as subcategories of $\D^b(X^{ss}(\cL_0)/G)$, where $\G^\pm_\bullet$ is the category identified with $\D^b(X^{ss}(\cL_\pm) / G)$ under restriction.
\end{ansatz}
Note that even when $\omega_X|_{X^{ss}(\cL_0)}$ is equivariantly trivial, the ansatz does not follow tautologically as it does for a balanced wall crossing, because the $Z^\pm_i$ are not identical but merely birational. Still one can verify the ansatz in some examples. For instance, one can recover a result of Segal \& Donovan \cite{DS12}:

\begin{ex}[Grassmannian flop] \label{ex_grassmannian}
Choose $k < N$ and let $V$ be a $k$-dimensional vector space. Consider the action of $G = GL(V)$ on $X = T^\ast \op{Hom}(V,\bC^N) = \op{Hom}(V,\bC^N) \times \op{Hom}(\bC^N,V)$ via $g \cdot(a,b) = (a g^{-1},gb)$. A \oneps\ $\lambda : \Gm \to G$ corresponds to a choice of weight decomposition $V \simeq \bigoplus V_\alpha$ under $\lambda$. A point $(a,b)$ has a limit under $\lambda$ iff
$$V_{>0} \subset \op{ker} (a) \quad \text{and} \quad \op{im} (b) \subset V_{\geq 0}$$
in which case the limit $(a_0,b_0)$ is the projection onto $V_0 \subset V$. There are only two nontrivial linearizations up to rational equivalence, $\cO_X(\det^{\pm})$. A point $(a,b)$ is semistable iff any \oneps\ for which $\lambda(t) \cdot (a,b)$ has a limit as $t\to 0$ has nonnegative pairing with the chosen character, $\det^\pm$.

In order to determine the stratification, it suffices to fix a maximal torus of $GL(V)$, i.e. and isomorphism $V \simeq \bC^k$, and to consider diagonal one parameter subgroups $(t^{w_1},\ldots,t^{w_k})$ with $w_1 \leq \cdots \leq w_k$. If we linearize with respect to $\det^{-1}$, then the KN stratification is indexed by $i = 0,\ldots,k-1$:
\begin{gather*}
\lambda_i = (1,\ldots,1,t,\ldots,t) \text{ with } i \text{ ones} \\
Z_i = \left\{ \left( \left[ \begin{array}{c|c} \square & 0 \end{array} \right], \begin{bmatrix} \ast \\ \hline 0 \end{bmatrix} \right) , \begin{array}{l} \text{with } \ast \in M_{i \times N}, \\ \text{and } \square \in M_{N\times i} \text{ full rank}  \end{array} \right\} \\
Y_i = \left\{ \left( \left[ \begin{array}{c|c} \square & 0 \end{array} \right], b \right) , \begin{array}{l} \text{with } b \in M_{k \times N} \text{ arbitrary}, \\ \text{and } \square \in M_{N\times i} \text{ full rank}  \end{array} \right\} \\
S_i = \{ (a,b) | b \text{ arbitrary, } \op{rank} a = i \}
\end{gather*}
So $(a,b) \in X$ is semistable iff $a$ is injective. If instead we linearize with respect to $\det$, then $(a,b)$ is semistable iff $b$ is surjective, the $\lambda_i$ flip, and the critical loci $Z_i$ are the same \emph{except that the role of $\square$ and $\ast$ reverse in the description of $Z_i$}. So this is an almost balanced wall crossing with $\cL_0 = \cO_X$ and $\cL^\prime = \cO_X(\det)$.

Let $\bG(k,N)$ be the Grassmannian parametrizing $k$-dimensional subspaces $V \subset \bC^N$, and let $0 \to U(k,N) \to \cO^N \to Q(k,N) \to 0$ be the tautological sequence of vector bundles on $\bG(k,N)$. Then $\X^{ss}( \det^{-1} )$ is the total space of $U(k,N)^N$, and $\X^{ss}(\det)$ is the total space of $(Q(N-k,N)^\dual)^N$ over $\bG(N-k,N)$.

In order to verify that $\G^+_w = \G^-_{w^\prime}$ for some $w^\prime$, one observes that the representations of $GL_k$ which form the Kapranov exceptional collection \cite{Ka84} lie in the weight windows for $G^+_0 \simeq \D^b(\X^{ss}(\det^{-1})) = \D^b(U(k,N)^N)$. Because $U(k,N)^N$ is a vector bundle over $\bG(k,N)$, these objects generate the derived category. One then verifies that these object lie in the weight windows for $\X^{ss}(\det)$ and generate this category for the same reason. Thus by verifying Ansatz \ref{ansatz_flop} we have established an equivalence of derived categories,
$$\D^b(U(k,N)^N) \simeq \D^b((Q(N-k,N)^\dual)^N).$$

The astute reader will observe that these two varieties are in fact isomorphic, but the derived equivalences we have constructed are natural in the sense that they generalize to families. Specifically, if $\cE$ is an $N$-dimensional vector bundle over a smooth variety $Y$, then the two GIT quotients of the total space of $\inner{\op{Hom}}(\cO_Y \otimes V , \cE) \oplus \inner{\op{Hom}}(\cE, \cO_Y \otimes V)$ by $GL(V)$ will have equivalent derived categories.
\end{ex}

The key to verifying Ansatz \ref{ansatz_flop} in this example was the simple geometry of the GIT quotients $\X^{ss}(\det^\pm)$ and the fact that we have explicit generators for the derived category of each. With a more detailed analysis, one can verify Ansatz \ref{ansatz_flop} for more examples of balanced wall crossings, and we will describe this in a future paper.
 
\begin{rem}
This example is similar to the generalized Mukai flops of \cite{CKL10}. The difference is that we are not restricting to the hyperk\"{a}hler moment fiber $\{ba = 0\}$. We will see in the next section that categorical Kirwan surjectivity applies in this example, but it is harder to verify Ansatz \ref{ansatz_flop}.
\end{rem}

\section{Applications to complete intersections: matrix factorizations and hyperk\"{a}hler reductions}
\label{sect_further_applications}

In the example of a projective variety, where we identified $\D^b(Y)$ with a full subcategory of the derived category of finitely generated graded modules over the homogeneous coordinate ring of $Y$, the vertex of the affine cone satisfied Property \hyperref[property_L_plus]{(L+)} ``for free." In more complicated examples, the cotangent positivity property \hyperref[property_L_plus]{(L+)} can be difficult to verify.

Here we discuss several techniques for extending categorical Kirwan surjectivity to stacks $X/G$ where $X$ is a local complete intersection. First we provide a geometric criterion for Property \hyperref[property_L_plus]{(L+)} to hold, which allows us to apply Theorem \ref{thm_kirwan_surj} to some hyperk\"{a}hler quotients. We also discuss two different approaches to categorical Kirwan surjectivity for LCI quotients, using Morita theory and derived categories of singularities.

\subsection{A criterion for Property \hyperref[property_L_plus]{(L+)} and non-abelian Hyperk\"{a}hler reduction}

In this section we study a particular setting in which Property \hyperref[property_L_plus]{(L+)} holds for the KN stratification of a singular quotient stack. This will allow us to address some hyperk\"{a}hler reductions by non-abelian groups. Throughout this section, if $V$ is a representation of a group $G$, and $\X = X/G$, then we will use $\cO_\X(V) \in \qcoh{\X}$ to denote the pullback of the quasicoherent sheaf on $\pt / G$ corresponding to $V$. Concretely, $\cO_{\X}(V)$ is the sheaf $\cO_X \otimes_k V$ with $G$-equivariant structure induced by the representation $V$.

Let $X^\prime$ be a smooth quasiprojective variety with a linearizable action of a reductive group $G$, and let $S^\prime = G \cdot Y^\prime \subset X^\prime$ be a closed KN stratum (Definition \ref{def_KN_stratification}). Because $X^\prime$ is smooth, $Y^\prime$ is a $P$-equivariant bundle of affine spaces over $Z^\prime$. Let $V$ be a linear representation of $G$, and $s : X^\prime \to V$ an equivariant map. Alternatively, we think of $s$ as an invariant global section of the locally free sheaf $\cO_{\X^\prime} (V)$. Note that if we decompose $V = V_+ \oplus V_0 \oplus V_-$ under the weights of $\lambda$, then $\Gamma(\S^\prime, \cO_{\S^\prime} ( V_- )) = 0$, so $s|_{\S^\prime}$ is a section of $\cO_{\S^\prime} ( V_0 \oplus V_+ )$.

\begin{prop} \label{prop_hyperkaehler_L_plus}
Let $X^\prime$, $S^\prime$, and $s$ be as above. Define $X = s^{-1}(0)$, $S = S^\prime \cap X$, and $Z = Z^\prime \cap X$, and assume that $X$ has codimension $\dim V$. If for all $z \in Z$, $(ds)_z : T_z X^\prime \to V$ is surjective in positive weights w.r.t. $\lambda$, then Property \hyperref[property_L_plus]{(L+)} holds for $S / G \hookrightarrow X / G$.
\end{prop}

Before proving the proposition, we compute the cotangent complex of $\S$.

\begin{lem} \label{lem_cotangent_complex_stratum}
If for all $z \in Z \subset Z^\prime$, $(ds)_z : T_z X^\prime \to V$ is surjective in positive weights w.r.t. $\lambda$, then
$$(\sigma^\ast \bL^\bdot_{\S})_{<0} \simeq [ \cO_\Z (V_+^\dual) \xrightarrow{(ds_+)^\dual} (\Omega_{Y^\prime}|_Z)_{<0} ]$$
and is thus a locally free sheaf concentrated in cohomological degree $0$.
\end{lem}
\begin{proof}
First of all note that from the inclusion $\sigma: \Z \hookrightarrow \S$ we have
$$(\sigma^\ast \bL^\bdot_\S)_{<0} \to (\bL^\bdot_\Z)_{<0} \to (\bL^\bdot_{\Z / \S})_{<0} \parr$$
The cotangent complex $\bL^\bdot_\Z$ is supported in weight $0$ because $\lambda$ acts trivially on $Z$, so the middle term vanishes, and we get $(\sigma^\ast \bL^\bdot_\S)_{<0} \simeq (\bL^\bdot_{\Z/\S})_{<0}[-1]$, so it suffices to consider the latter.

By definition $Y$ is the zero fiber of $s : Y^\prime \to V_0 \oplus V_+$. Denote by $s_0$ the section of $V_0$ induced by the projection of $P$-modules $V_+ \oplus V_0 \to V_0$. We consider the intermediate variety $Y \subset Y_0 := s_0^{-1} (0) \subset Y^\prime$. Note that $Y_0 = \pi^{-1} (Z_0)$, where $\pi :Y^\prime \to Z^\prime$ is the projection.

Note that $Y_0 \to Z$ is a bundle of affine spaces with section $\sigma$, so in particular $\Z \subset \S_0$ is a regular embedding with conormal bundle $(\Omega^1_{Y^\prime}|_Z)_{<0} = (\Omega^1_{X^\prime}|_Z)_{<0}$. Furthermore, on $Y_0$ the section $s_0$ vanishes by construction, so $Y \subset Y_0$, which by definition is the vanishing locus of $s|_{Y_0}$, is actually the vanishing locus of the map $s_+ : Y_0 \to V_+$. The surjectivity of $(ds)_z$ for $z \in Z$ in positive weights implies that $s_+^{-1}(0)$ has expected codimension in every fiber over $Z$ and thus $\S \subset \S_0$ is a regular embedding with conormal bundle $\cO_{\S} (V_+^\dual)$.

It now follows from the canonical triangle for $\Z \subset \S \subset \S_0$ that
$$\bL^\bdot_{\Z/\S} \simeq \op{Cone}(\sigma^\ast \bL_{\S / \S_0} \to \bL_{\Z/\S_0}) \simeq [\cO_{\Z} (V_+^\dual) \xrightarrow{ds_+} (\Omega^1_{Y^\prime}|_Z)_{<0}]$$
with terms concentrated in cohomological degree $-2$ and $-1$. The result follows.
\end{proof}

\begin{proof}[Proof of Proposition \ref{prop_hyperkaehler_L_plus}]

We will use Lemma \ref{lem_cotangent_complex_stratum} to compute the restriction of the relative cotangent complex, $(\sigma^\ast \bL^\bdot_{\S / \X})_{<0}$. We consider the canonical diagram
$$\xymatrix{ [\cO_Y(V^\dual) \to \Omega^1_{X^\prime}|_Y] \ar[r] \ar[d]^a & [\cO_Y(V_{\geq 0})^\dual \to \Omega^1_{Y^\prime}|_{Y}] \ar[d]^b & & \\ j^\ast \bL^\bdot_{\X} \ar[r] & \bL^\bdot_{\S} \ar[r] & \bL^\bdot_{\S / \X} \ar@{-->}[r] & }$$
where the bottom row is an exact triangle and we have used the identification $\S^\prime \simeq Y^\prime / P$ and $\S \simeq Y/P$. Because $X \subset X^\prime$ has the expected codimension, it is a complete intersection and the morphism $a$ is a quasi-isomorphism. Lemma \ref{lem_cotangent_complex_stratum} implies that $b$ is a quasi-isomorphism after applying the functor $(\sigma^\ast (\bullet))_{<0}$.

Thus we have a quasi-isomorphism
\begin{align*}
(\sigma^\ast \bL^\bdot_{\S/\X})_{<0} &\simeq \op{Cone} \left( [\cO_Z(V^\dual) \to \Omega^1_{X^\prime}|_Z] \to [\cO_Z(V_{\geq 0})^\dual \to \Omega^1_{Y^\prime}|_{Z}] \right)_{<0}\\
&\simeq \op{Cone} \left( (\Omega^1_{X^\prime}|_Z)_{<0} \to (\Omega^1_{Y^\prime}|_{Z})_{<0} \right) \simeq 0
\end{align*}
The last isomorphism follows because $\Omega^1_{Y^\prime}|_Z$ is the negative weight eigenspace of $\Omega^1_{X^\prime}|_Z$ by construction.
\end{proof}

Now let $(M,\omega)$ be an algebraic symplectic manifold with a Hamiltonian $G$ action, i.e. there is a $G$-equivariant algebraic map $\mu : M \to \lie{g}^\dual$ such that for any $\xi \in \lie{g}$, $d \langle \xi , \mu \rangle = - \omega (\partial_\xi, \bullet) \in \Gamma(M,\Omega^1_M)$, where $\partial_\xi$ is the vector field corresponding to $\xi \in \lie{g}$.

For any point $x \in M$, let $G_x$ be the stabilizer of $x$. We have an exact sequence
\begin{equation} \label{eqn_moment_sequence}
0 \to \op{Lie} G_x \to \lie{g} \xrightarrow{d\mu} T^\ast_x M \overset{\omega}{\simeq} T_x M \to (\fN_{G \cdot x} M)_x \to 0,
\end{equation}
showing that $X := \mu^{-1}(0)$ is regular at any point with finite stabilizer groups. Thus if the set of such points is dense in $X$, then $X \subset M$ is a complete intersection cut out by $\mu$.
\begin{cor}
Let $(M,\omega)$ be a projective-over-affine algebraic symplectic manifold with a Hamiltonian action of a reductive group $G$, with a choice of linearization, and let $X = \mu^{-1}(0) \subset M$. If $X^{s}$ is dense in $X$, then Property \hyperref[property_L_plus]{(L+)} holds for the GIT stratification of $X$.
\end{cor}

\begin{proof}
The exact sequence \eqref{eqn_moment_sequence} shows that $d\mu$ is injective at any point of $X$ which has a finite stabilizer, and in particular any point of $X^s$. Hence $X$ has the expected codimension. Furthermore, for any $z \in Z$, we have $$\op{Lie}(G_z)^\dual = \op{coker}( (d\mu)_z : T_z M \to \lie{g}^\dual).$$
For any point in a KN stratum $p \in S \subset M$, Property \ref{property_S_2} implies that $\op{Lie}(G_p) \subset \op{Lie}(P_\lambda)$, which has nonnegative weights with respect to the adjoint action of $\lambda$. It follows that $\mu : M \to \lie{g}^\dual$ satisfies the hypotheses of Proposition \ref{prop_hyperkaehler_L_plus}.
\end{proof}

\begin{ex}[stratified Mukai flop]
In Example \ref{ex_grassmannian} we considered the GIT stratification for the action of $GL(V)$ on $M := \op{Hom}(V,\bC^N) \times \op{Hom}(\bC^N,V)$. This representation is symplectic, and it has an algebraic moment map $\mu(a,b) = ba \in \lie{gl}(V)$. The KN stratification of $X = \mu^{-1}(0)$ is induced by the stratification of $M$. Thus $Y_i \subset X$ consists of
$$Y_i = \left\{ \left( \left[ \begin{array}{c|c} a_1 & 0 \end{array} \right], \begin{bmatrix} b_1 \\ \hline b_2 \end{bmatrix} \right) , \begin{array}{l} \text{with } b_1 a_1 = 0, b_2 a_1 = 0 , \\ \text{and } a_1 \in M_{N\times i} \text{ full rank}  \end{array} \right\} \\$$
and $Z_i \subset Y_i$ are those points where $b_2=0$. Note that over a point in $Z_i$, the condition $b_2 a_1 = 0$ is linear in the fiber, and so $Y_i \to Z_i$ satisfies Property \hyperref[property_A]{(A)}.

The GIT quotient $X^{ss}/GL(V)$ is the cotangent bundle $T^\ast \bG(k,N)$. Property \hyperref[property_A]{(A)} holds in this example, and Property \hyperref[property_L_plus]{(L+)} holds by Proposition \ref{prop_hyperkaehler_L_plus}, so Theorem \ref{thm_kirwan_surj} gives a fully faithful embedding $\D^b(T^\ast \bG(k,N)) \subset \D^b(X / GL(V))$ for any choice of integers $w_i$. The derived category $\D^b(T^\ast \bG(k,N))$ has been intensely studied by Cautis, Kamnitzer, and Licata from the perspective of categorical $\lie{sl}_2$ actions. We will discuss the connection between their results and categorical Kirwan surjectivity in future work.
\end{ex}

\subsection{Applications to derived categories of singularities and abelian hyperk\"{a}hler reductions, via Morita theory}
In this section we discuss an alternative method of extending Theorem \ref{thm_kirwan_surj_prelim} to singular stacks which are complete intersections in the sense that they arise as the derived fibers of maps between smooth stacks. We use derived Morita theory as developed in \cite{BFN10}, which applies in our context because all of our stacks are global quotients by linearizable group actions in characteristic $0$, and hence they are perfect.

If $f : \X \to \fB$ is a morphism between perfect stacks and $b \in \fB(k)$ a closed point, the fiber $\X_b := \X \times^L_\fB \{b\}$ is a derived stack and should be understood in the context of derived algebraic geometry. However, one can describe the categories $\D^b(\X_b)$ and $\op{Perf}(\X_b)$ using less technology. If $\cA^\bdot \to \cO_{b}$ is a resolution of $\cO_b$ by a commutative dg-$\cO_{\fB}$-algebra which is locally freely generated as a graded commutative algebra, then $f^\ast \cA^\bdot$ is a sheaf of commutative dg-algebras over $\X$. $\D^b(\X_b)$ is equivalent to the derived category of sheaves of dg-modules over $f^\ast \cA^\bdot$ whose cohomology sheaves are coherent over $\X$, and $\op{Perf}(\X)$ is the category of perfect complexes of $f^\ast \cA^\bdot$-modules.

\begin{prop} \label{prop_derived_fibers}
Let $\X = X/G$ be a smooth quotient stack, and let $\X^{ss} \subset \X$ be an open substack whose complement admits a KN stratification. Let $f : \X \to \fB$ be a morphism to a perfect stack $\fB$, and let $b \in \fB(k)$ be a closed $k$-point. Assume that for each $Z_i$ and $\lambda_i$ in the KN stratification of $\X^{us}$ and each point $x \in Z_i$, the homomorphism $\Gm \to \op{Aut}_{\X} (z) \to \op{Aut}_{\fB} (f(z))$ is trivial. Then splitting of Theorem \ref{thm_kirwan_surj_prelim} induces splittings of the natural restriction functors:
$$\xymatrix@R=5pt{\D^b(\X_b) \ar@{->>}[r]_{i^\ast} & \D^b(\X^{ss}_b) \ar@/_/[l] \\ \op{Perf}(\X_b) \ar@{->>}[r]_{i^\ast} & \op{Perf}(\X^{ss}_b) \ar@/_/[l]}$$
\end{prop}

\begin{rem}
When $\fB = \bA^r$ and $b = 0 \in \fB$, and the fiber over $b$ has codimension $r$, then the derived fiber agrees with the classical fiber. Hence the conclusion of the proposition is purely classical in this complete intersection case.
\end{rem}

\begin{proof}
The restriction $i^\ast : \op{Perf}(\X) \to \op{Perf}(\X^{ss})$ is symmetric monoidal dg-functor, and it is canonically a functor of module categories over the symmetric monoidal dg-category $\op{Perf}(\fB)^\otimes$. The subcategory $\G_w \subset \D^b(\X)$ used to construct the splitting in Theorem \ref{thm_kirwan_surj_prelim} is defined through conditions on the weights of $F|_{Z_i}$ with respect to the $\lambda_i$. If the homomorphism $\Gm \to \op{Aut}_{\fB}(f(z))$ is trivial, then for any $E^\bdot \in \op{Perf}(\fB)$, we have $f^\ast E^\bdot \otimes \G_w \subset \G_w$. It follows that $\G_w$ is canonically a $\op{Perf}(\fB)^\otimes$-module subcategory, and the splitting constructed in Theorem \ref{thm_kirwan_surj_prelim} is a splitting as module categories over $\op{Perf}(\fB)^\otimes$.

Because the splitting is $\op{Perf}(\fB)$-linear, the restriction functor
$$\op{Fun}_{\op{Perf}(\fB)} \left( \op{Perf}(\{b\}), \op{Perf}(\X) \right) \twoheadrightarrow \op{Fun}_{\op{Perf}(\fB)} \left( \op{Perf}(\{b\}), \op{Perf}(\X^{ss}) \right)$$
admits a splitting as well. We claim that the source (respectively target) category can be canonically identified with $\D^b(\X_b)$ (respectively $\D^b(\X^{ss}_b)$), and it thus follows that $\D^b(\X_b) \to \D^b(\X^{ss}_b)$ admits a splitting.

For any diagram of perfect stacks $X^\prime \to X \leftarrow Y$, one has a fully faithful embedding 
$$\op{Fun}_{\op{Perf}(X)}(\op{Perf}(X^\prime),\op{Perf}(Y)) \subset \op{Fun}_{\D_{qc}(X)}(\D_{qc}(X^\prime),\D_{qc}(Y)),$$
where the latter denotes continuous $\D_{qc}(X)$-linear functors. As a result of Theorem 1.2 of \cite{BFN10}, we can identify the latter functor category with $\D_{qc}(X^\prime \times_X Y)$. Explicitly in our situation, $\op{Fun}_{\op{Perf}(\fB)} \left( \op{Perf}(\{b\}), \op{Perf}(\X) \right)$ is equivalent to the full dg-subcategory of $D(\X_b)$ consisting of complexes such that the corresponding integral transform $\D_{qc}(\{b\}) \to \D_{qc}(\X)$ preserves perfect complexes. Because $\op{Perf}(\{b\})$ is generated by the structure sheaf, this is equivalent to the subcategory of $\D_{qc}(\X_b)$ of objects whose pushforward to $\X$ is perfect. Because $\X$ is smooth and $\X_b \subset \X$ is a closed substack, this is precisely $\D^b(\X_b)$.

We can apply a similar analysis to then tensor product. The $\op{Perf}(\fB)$-linearity of the splitting of $i^\ast$ implies that the restriction functor
$$\op{Perf}(\{b\}) \otimes_{\op{Perf}(\fB)} \op{Perf}(\X) \to \op{Perf}(\{b\}) \otimes_{\op{Perf}(\fB)} \op{Perf}(\X^{ss})$$
admits a splitting as well. These categories are identified, by Theorem 1.2 of \cite{BFN10}, with $\D_{qc}(\X_b)^c = \op{Perf}(\X_b)$ and $\D_{qc}(\X^{ss}_b)^c = \op{Perf}(\X^{ss}_b)$ respectively. 
\end{proof}

\begin{ex}
As a special case of Proposition \ref{prop_derived_fibers}, one obtains equivalences of derived categories of singularities. Namely, if $W : \X \to \bA^1$ is a function, sometimes referred to as a super potential, then the category of singularities corresponding to $W$ is
$$\op{B}(\X,W) := \D^b_{sing} (W^{-1}(0)) = \D^b (W^{-1}(0)) / \op{Perf}(W^{-1}(0))$$
Proposition \ref{prop_derived_fibers} implies that the restriction functor $\op{B}(\X,W) \to \op{B}(\X^{ss},W|_{\X^{ss}})$ splits. If for GIT quotients of $\X$ corresponding to two different linearizations, $\op{Perf}(\X^{ss}(\cL_1))$ and $\op{Perf}(\X^{ss}(\cL_2))$ can be identified with the same subcategory  of $\op{Perf}(\X)$ as in Proposition \ref{prop_birat_cobord}, then the corresponding categories of singularities are equivalent
$$\op{B}(\X^{ss}(\cL_1),W|_{\X^{ss}(\cL_1)}) \simeq \op{B}(\X^{ss}(\cL_2),W|_{\X^{ss}(\cL_2)})$$
Note that by an equivariant generalization of \cite{Or12}, these results could be equivalently formulated in terms of categories of matrix factorizations.

In addition, if we introduce an auxiliary $\Gm$ action on $X$ with respect to which $W \in \Gamma(\cO_X)$ has weight $2$, then the categories $\op{B}(X/G \times \Gm, W)$ are often referred to as graded categories of singularities \cite{Sh12}. Proposition \ref{prop_derived_fibers} applies in this situation as well, where $W$ is interpreted as a morphism $X /G \times \Gm \to \bA^1 / \Gm$.
\end{ex}

Proposition \ref{prop_derived_fibers} also applies to the context of hyperk\"{a}hler reduction. Let $T$ be a torus, or any group whose connected component is a torus, and consider a Hamiltonian action of $T$ on a hyperk\"{a}hler manifold $X$, or more generally an algebraic symplectic manifold, with algebraic moment map $\mu : X/T \to \lie{t}^\dual$. One forms the hyperk\"{a}hler quotient by choosing a linearization on $X/T$ and defining $X///T = \mu^{-1}(0) \cap \X^{ss}$. Thus we are in the setting of Proposition \ref{prop_derived_fibers}.

\begin{cor} \label{cor_hyperkahler}
Let $T$ be an extension of a finite group by a torus. Let $T$ act on an algebraic symplectic manifold $X$ with algebraic moment map $\mu : X \to \lie{t}^\dual$. Then the restriction functors
\begin{gather*}
\D^b(\mu^{-1}(0)/T) \to \D^b(\mu^{-1}(0)^{ss} / T) \\
\op{Perf}(\mu^{-1}(0)/T) \to \op{Perf}(\mu^{-1}(0)^{ss} / T)
\end{gather*}
both split, assuming that $\mu^{-1}(0)$ has the expected codimension.
\end{cor}
\begin{proof}
Because the adjoint representation of the connected component of $T$ on $\lie{t}$ is trivial, the condition on automorphism groups in Proposition \ref{prop_derived_fibers} holds automatically for moment map $\mu : X / T \to \lie{t} / T$. The fact that $\mu^{-1}(0)$ has the expected codimension implies that the derived fiber and classical fiber agree.
\end{proof}

Because $\mu$ is a moment map for the $T$-action, the fiber $\mu^{-1}(0)$ will have expected codimension whenever $X^{ss} = X^s$. Then $T$ is abelian this will be the case for generic linearizations. Thus we have a hyperk\"{a}hler analog of Corollary \ref{cor_all_quotients_equivalent}.

\begin{cor}
Let $X$ be a projective-over-affine algebraic symplectic manifold with a Hamiltonian action of a torus $T$. Then the hyperk\"{a}hler quotient with respect to any two generic linearizations are derived-equivalent.
\end{cor}
\begin{proof}
By Corollary \ref{cor_all_quotients_equivalent} all $\X^{ss}(\cL)$ for generic $\cL$ will be derived equivalent. More precisely there will be a finite sequence of wall crossings connecting any two generic linearizations such that for each wall crossing one can identify a subcategory of $\op{Perf}(\X^{ss}(\cL_0))$ which maps isomorphically, via restriction, to both $\op{Perf}(\X^{ss}(\cL_\pm))$. The wall crossings in this case are truly faithful, in the sense of \cite{DH98}, and balanced in the sense discussed above. Thus the loci of points of $X^{ss}(\cL_0)$ which have positive dimensional stabilizers is the disjoint union of the $Z_i$ (which are codimension at least $2$), and their stabilizers are exactly the $\lambda_i(\Gm)$. Using this one can show that $\mu^{-1}(0)$ has the expected codimension in $X^{ss}(\cL_0)$. Hence we can apply Corollary \ref{cor_hyperkahler} to conclude that the splittings $\op{Perf}(\X^{ss}(\cL_\pm)) \subset \op{Perf}(\X^{ss}(\cL_0))$ descend to $\mu^{-1}(0)$, giving an equivalence $\op{Perf}(\X^{ss}(\cL_+)) \simeq \op{Perf}(\X^{ss}(\cL_-))$.
\end{proof}

\bibliography{references}{}
\bibliographystyle{plain}

\end{document}